\documentclass[11pt,english,reqno]{smfart}
\usepackage[letterpaper,left=1in,right=1in,top=1.5in,bottom=1.5in]{geometry}

\usepackage[T1]{fontenc}
\usepackage[english,francais]{babel}

\usepackage{amssymb,url,xspace,smfthm}

\newcommand{\SMF}{Soci\'et\'e Ma\-th\'e\-Ma\-ti\-que de France}
\newcommand{\BibTeX}{{\scshape Bib}\kern-.08em\TeX}
\newcommand{\T}{\S\kern .15em\relax }
\newcommand{\AMS}{$\mathcal{A}$\kern-.1667em\lower.5ex\hbox
        {$\mathcal{M}$}\kern-.125em$\mathcal{S}$}

\usepackage[hidelinks]{hyperref}
\usepackage[
]{cleveref}

\newcommand*{\TitleFont}{%
      \usefont{\encodingdefault}{\rmdefault}{}{n}%
      \fontsize{9}{20}%
      \selectfont}

\title[Zeta and wave operators]{Zeta zeros and prolate wave operators\\\TitleFont{Semilocal adelic  operators}}\date {}
\author{Alain Connes}
\address{Coll\`ege de France\\3 Rue d'Ulm\\75005 Paris, France}
\address{I.H.E.S., France}
\email{\url{alain@connes.org}}
\author{Caterina Consani}
\address{Department of Mathematics\\Johns Hopkins University\\Baltimore, MD 21218, USA}
\email{\url{cconsan1@jhu.edu}}
\author{Henri Moscovici}
\address{Department of Mathematics\\Ohio State University\\Columbus, OH 43210, USA}
\email{\url{moscovici.1@osu.edu}}

\subjclass{\href{http://www.ams.org/mathscinet/msc/msc2020.html?t=11Mxx&btn=Current}{11M06},
\href{http://www.ams.org/mathscinet/msc/msc2020.html?t=11Mxx&btn=Current}{11M55},
\href{http://www.ams.org/mathscinet/msc/msc2020.html?t=22E46&btn=Current}{22E46},
\href{http://www.ams.org/mathscinet/msc/msc2020.html?t=33D45&btn=Current}{33D45},
\href{http://www.ams.org/mathscinet/msc/msc2020.html?t=34B20&btn=Current}{34B20}}

\usepackage{amsmath}
\usepackage{amscd}
\usepackage{amsbsy}
\usepackage{amssymb}
\usepackage{verbatim}
\usepackage{eufrak}
\usepackage{eucal}
\usepackage{microtype}
\usepackage{mathrsfs}
\usepackage{amsthm}
\usepackage{stmaryrd}
\usepackage[all,cmtip]{xy}
\usepackage{tikz}
\usetikzlibrary{matrix,arrows}
\usepackage[abbrev,lite,alphabetic]{amsrefs}
\usepackage{dsfont}

\usepackage{color}
\definecolor{todo}{rgb}{1,0,0}
\definecolor{conditional}{rgb}{0,1,0}
\definecolor{e-mail}{rgb}{0,.40,.80}
\definecolor{reference}{rgb}{.20,.60,.22}
\definecolor{mrnumber}{rgb}{.80,.40,0} 
\definecolor{citation}{rgb}{0,.40,.80} 

\newcommand{\ie}{{\it i.e.\/}\ }
\newcommand{\qqq}{\,,\,~\forall}
\newcommand{\cf}{{\it cf.}}

\newcommand{\htt}{{Hardy-Titchmarsh~}}
\newcommand{\fourierer}{\fourier_{e_\R}}
\newcommand{\fep}{{\fourier_{e_p}}}
\newcommand{\fourier}{{\F}}

\newcommand{\son}{{\bf S}}

\newcommand{\zp}{{1_{\Z_p}}}

\newcommand{\GL}{{\rm GL}}

\newcommand{\cA}{\mathcal{A}}
\newcommand{\cB}{{\mathcal B}}
\newcommand{\cE}{\mathcal{E}}
\newcommand{\cH}{\mathcal{H}}
\newcommand{\cK}{\mathcal{K}}
\newcommand{\cM}{\mathcal{M}}
\newcommand{\cP}{{\mathcal P}}
\newcommand{\cS}{\mathcal{S}}
\newcommand{\cU}{{\mathcal U}}
\newcommand{\cV}{{\mathcal V}}
\newcommand{\cW}{{\mathcal W}}

\newcommand{\A}{{\mathbb A}}
\newcommand{\C}{{\mathbb C}}
\newcommand{\F}{{\mathbb F}}
\newcommand{\bH}{{\mathbb H}}
\newcommand{\K}{{\mathbb K}}
\newcommand{\N}{\mathbb{N}}
\newcommand{\Q}{{\mathbb Q}}
\newcommand{\R}{\mathbb{R}}
\newcommand{\scal}{{\mathbb S}}
\newcommand{\Z}{\mathbb{Z}}

\theoremstyle{plain}
\newtheorem{thm}{Theorem}[section]
\newtheorem{defn}[thm]{Definition}
\newtheorem{cor}[thm]{Corollary}
\newtheorem{lem}[thm]{Lemma}
\newtheorem{rem}[thm]{Remark}

\newtheorem*{theorem1}{Theorem 1}
\newtheorem*{theorem2}{Theorem 2}

\begin{document}
\def\smfbyname{}

\begin{abstract}
We integrate in the framework  of the semilocal trace formula   two recent discoveries on the spectral realization of the zeros of the Riemann zeta function by introducing a semilocal analogue of the prolate wave operator. The latter plays a key role both in the spectral realization of the low lying zeros of zeta--using the positive part of its spectrum--and of their ultraviolet behavior--using the Sonin space which corresponds to the negative part of the spectrum. In the archimedean case  the prolate operator is the sum of the square of the  scaling operator with the grading of orthogonal polynomials, and we show  that this formulation extends to the semilocal case.   We prove the stability of the semilocal Sonin space under the increase of the finite set of places which govern the semilocal framework and describe their relation with Hilbert spaces of entire functions.  Finally, we relate the prolate operator to the metaplectic representation of the double cover of $\operatorname{SL}(2,\R)$ with the goal of obtaining (in a forthcoming paper) a second candidate for the semilocal prolate operator.

\end{abstract}
\maketitle

\tableofcontents
\section{Introduction}

The difficulty of solving the Riemann Hypothesis (RH) is often  
primarily attributed to the infinite number of terms within the Euler product.
\begin{equation}\label{eulerproduct1}
 	\zeta(s) = \prod_p (1- p^{-s})^{-1}
 \end{equation} 
However, contrary to this widespread belief,  there exists a property 
$P(n)$, involving only
the Euler factors for primes smaller than $n$, and
 whose validity for all $n$ is equivalent to RH. 
This is derived  from  Weil's positivity criterion which involves the quadratic
 form $Q_n$ defined using the Riemann-Weil explicit formulas applied to test functions with support in the compact symmetric interval $[\frac 1n,n]$.  Furthermore,  the semilocal trace formula of \cite{C98} gives, for each n, a Hilbert space theoretic framework in which the Weil
quadratic form $Q_n$ becomes the trace of a simple operator theoretic expression, thus offering a natural ground of exploration in order to
attack RH.
The present paper furthers the development of the operator-theoretic approach to RH  
by providing a unified framework that incorporates two recent discoveries related to the spectral realization of the zeros of the Riemann zeta function.
\newline
On the one hand, the results of \cite{CCjot,CCweil} show that the semilocal trace formula  of \cite{C98} gives a conceptual explanation of Weil's positivity for the archimedean place,  the source of it being the Hilbert space representation of the scaling group. Moreover, in the same vein \cite{CCVJ} shows that one can access to the infrared (low lying)  part of the zeros of the Riemann zeta function using  the scaling operator with periodic boundary conditions, restricted to the orthogonal of the radical of the Weil quadratic form. On the other hand the  unusual ultraviolet behavior of the zeros of the Riemann zeta function found an unexpected spectral incarnation in \cite{CM}  in terms of the \emph{negative} eigenvalues of the  classical prolate wave operator extended to the whole real line.\newline
In this paper we blend together these developments and go beyond the single archimedean place, using the semilocal framework of \cite{C98} which provides a canonical Hilbert space-operator theoretic stage in which the main players continue to make sense, and both the explicit formula and the radical of the Weil quadratic form have a conceptual meaning.

\noindent The role of the prolate operator is crucial in both (infrared and ultraviolet) observed agreements with zeros of zeta. Its significance in \cite{CCVJ} arises because the radical of the Weil quadratic form is given by the range of the map $\cE$ defined on the Schwartz space $\cS(\R)_0^{\rm ev}$ by the formula $\cE(f)(u):=u^{1/2}\sum f(nu)$. When applied to eigenfunctions $\varphi_n$ of the prolate operator for \emph{positive} eigenvalues, one generates extremely small eigenvalues of the Weil quadratic form restricted to functions with fixed compact support. The corresponding eigenfunctions $\psi_n$ of the Weil quadratic form are reconstructed in \cite{CCVJ} by Gram-Schmidt orthogonalization of $\cE(\varphi_n)$.  The semilocal form of the map $\cE$  agrees with the canonical identification of the semilocal functions with functions on idele classes. The conceptual meaning of the map $\cE$ will become clear in the semilocal formalism. The eigenfunctions $\psi_n$ were used in  \cite{CCVJ} to condition the scaling operator with  periodic boundary conditions. With this process we generated the low-lying zeros of the zeta function, while the ultraviolet behavior of the zeros remained out of reach. Thus the fact that the expected  UV behavior  is realized by the negative spectrum of the prolate operator suggests the tantalizing program  of finding the semilocal analogue of the prolate operator.  We expect that the use of such operator-theoretic tools in the semilocal case opens a way to handle Weil's positivity as in \cite{CCweil}. In fact, the operator theoretic aspect of the present paper provides a more precise
strategy for addressing the semilocal Weil positivity by comparing the trace functional associated to the operator- which is automatically positive for a selfadjoint operator- with the Weil functional. The conditioning, by the radical of the Weil quadratic form, that worked for the scaling operator in the infrared case will be implemented automatically by the orthogonality of the positive and the negative part of the spectrum of the semilocal prolate operator, whose corresponding negative eigenspace\footnote{up to a finite dimensional possible discrepancy} was identified in  \cite{CM}  to the Sonin space.   

Our approach involves reinterpreting 
 the prolate operator to propose  a candidate for its semilocal counterpart.  This is done by placing at the forefront the scaling operator,  whose  semilocal version is well established in \cite{C98},  as well as  the Sonin space.  \vspace{.05in}

The scaling operator $\scal$ is the selfadjoint generator of the group of dilations in the Hilbert space  $L^2(\R)^{ev}$ of square integrable even functions on the real line. The prolate operator admits a simple expression in terms of the operator $\scal$ and the vector  given by the function $\xi(x):=e^{-\pi x^2}$. This vector belongs to the domain of $\scal^n$ for any $n$, and the span of the $\scal^n \xi$ is $L^2(\R)^{ev}$. We let $N$ be the grading associated to the filtration coming from the linear span of polynomials $p(\scal)\xi$, with $\operatorname{deg} P \leq n$. Ignoring the delicate domain definition needed to obtain a selfadjoint operator, the prolate operator ${\bf W}_\lambda$ is given by the formal expression
\begin{equation}\label{wlambda}
{\bf W}_\lambda=-\scal^2+2\pi\lambda^2(4N+1)-\frac 14.
\end{equation}
Formula \eqref{wlambda} is meaningful for any {\em cyclic pair} $(D, \xi)$ given by  a selfadjoint operator $D$ acting in a Hilbert space $\cH$ and a unit vector $\xi \in \cap_{n\in \N} {\rm Dom} D^n$ which is  cyclic  (\ie such that the vectors $p(D)\xi$ with $p$ a polynomial, are dense  in  $\cH$).  For any   cyclic pair $(D, \xi)$, the spectral theorem provides a {\em canonical form} (or diagonalization) in which $D$ is represented as the multiplication by the variable $s\in \R$ in the Hilbert space $\cH=L^2(\R,d\mu)$  where $d\mu$ is the spectral measure defined by  $\int f(s)d\mu(s):=\langle f(D)\xi\vert \xi\rangle$,   and the cyclic vector $\xi$ is the constant function $1$. The theory of orthogonal polynomials for the measure $d\mu$ then provides the representation of $D$ as a Jacobi matrix and of the grading operator $N$ as a diagonal matrix. We recall these basic facts in  \cref{cyclpair}.  It turns out that one can naturally associate to a cyclic pair a spectral triple  $(\cA,\cH,D)$ where $\cH$ and $D$ are unchanged while the algebra $\cA:=c_0(\N)$ acts via the grading operator $N$ associated to the orthogonal polynomials of the spectral measure. The role of the spectral triple  $(\cA,\cH,D)$ is to understand algebraically  the prolate operator as a perturbation of the operator $-D^2$ by the addition of an element of the algebra $\cA$. The spectral triple  $(\cA,\cH,D)$ of a cyclic pair is {\em even} (in the sense that it admits a $\Z/2\Z$-grading commuting with the algebra and anticommuting with $D$) if and only if the spectral measure $d\mu$ is even, \ie invariant under $s\mapsto -s$. Moreover such a $\Z/2\Z$-grading $\gamma$ is unique if it exists.   In the even case one obtains a Jacobi matrix representation for the prolate operator (\cref{jacobprol}).\newline
In  \cref{pso} we describe in  \Cref{mainstart} the canonical form of the even cyclic pair associated to the scaling operator $D=\scal$ and the cyclic vector  $\xi_\infty(x)=2^{1/4}e^{-\pi x^2}$.  The grading is given by the Fourier transform $\fourier_{e_\R}$ acting in    $ L^2(\R)^{ev}$. The isomorphism with the canonical form is given by the unitary 
\begin{equation}\label{htintro}
	\cV=\cM\circ\cU:L^2(\R)^{ev}\to L^2(\R,dm)
\end{equation}
which is the composition of the unitary transformation $\cU: L^2(\R)^{ev}\to L^2(\R)$
$$
\cU(f)=\pi^{-\frac12}\ \int_0^{\infty} f(v) v^{\frac 12-is}d^*v\,.
$$ 
with the  multiplication operator  $\cM(f)(s):=\cU(\xi_\infty)^{-1}(s)f(s)$ which is a unitary isomorphism $\cM:L^2(\R)\to L^2(\R,dm)$ for the probability measure $dm(s):=(2\pi)^{-\frac 32}\vert \Gamma(\frac 14+ \frac{is}{2})\vert^2ds$ on $\R$. The unitary isomorphism $\cV$ of \eqref{htintro} is the \htt transform  which, modulo normalizations, appears already in the paper \cite[Theorem 1, p. 201]{HT}. 
Their original motivation was  to construct self-reciprocal functions by conjugating the Fourier transform with the symmetry $s\mapsto -s$. This puts in evidence the role of the evenness, \ie the $\Z/2\Z$-grading,  of the cyclic pair. The  unitary isomorphism $\cV$ transforms the Hermite functions  into the orthogonal polynomials $P_n$ for the measure $dm$ and the scaling operator $\scal$ into the multiplication by $s$. Moreover the prolate wave operator 
\begin{equation}
({\bf W}_{\lambda}\psi)(q) = \,- \partial (( \lambda^2- q^2) \partial )\,
\psi(q) + (2 \pi \lambda  q)^2 \, \psi(q)   \label{WLambdaq1}
\end{equation} 
becomes a special case of \eqref{wlambda} and, in fact, of a general construction for orthogonal polynomials. More precisely, with $N$  the grading operator  $N(P_n):=n\ P_n$  in $L^2(\R,dm)$, one obtains
		\begin{equation}\label{wlambdconjintro}
\cV\circ {\bf W}_{\lambda}\circ \cV^*=-s^2+2\pi\lambda^2(4N+1)-\frac 14
\end{equation}
	 where $s^2$ is the operator of multiplication by the square of the variable in $L^2(\R,dm)$.  The right hand side of \eqref{wlambdconjintro} is meaningful  in the general context of orthogonal polynomials.
	 We compute in \cref{sectj1} the hermitian  Jacobi matrix of the scaling operator $\scal$ in the basis  of orthogonal polynomials. Its diagonal elements are $0$, while the non-zero matrix elements above the diagonal are 
 $\scal_{n,n+1}=i\sqrt{(n+1/2)(n+1)}$. In \cref{jacofan} we describe the 
 two hermitian Jacobi matrices $J_\pm$ associated to the operators  $\scal^2$ restricted to the eigenspaces of $\fourier_{e_\R}$ for the eigenvalues $\pm 1$. Similarly we obtain the Jacobi matrices of  the prolate operator ${\bf W}_{\lambda}$.    \newline
   In \cref{sectmapE} we show that the zeros of the zeta function appear from the action of the scaling operator on the quotient by the range of the map $\cE$ applied to the positive eigenspace of ${\bf W}_{\lambda}$ when $\lambda\to \infty$.\vspace{0.03in}
      
\Cref{sectsemilocal} is devoted to the extension of the \htt transform in the semilocal situation involving a finite set $S$ of places of $\Q$ containing the archimedean one, and the analysis of the semilocal Sonin space. The semilocal version of \Cref{mainstart}  involves a new measure which, up to normalization, is given  by the square of the  absolute value of the restriction to the critical line of the product of local factors for the places in $S$.\newline
 To state this result we need the following notation (see \cref{notat}). The ring of semilocal adeles is the product $\A_S=\prod_S \Q_v$,  
 the semilocal adele class space is the quotient $X_S:=\A_S/\Gamma$, 
  where  $$\Gamma=\{ \pm p_1^{n_1} \cdots p_k^{n_k} \, :\,  p_j
\in S \setminus\{ \infty \} \,,\, n_j\in \Z\}\subset \GL_1(\A_{S})= \prod_{p\in S} \GL_1(\Q_p). $$
The semilocal idele class group $C_S:=\GL_1(\A_{S})/\Gamma$ acts on $X_S$ by multiplication and this action restricts to its maximal compact subgroup $K_S$. One has a canonical isomorphism $w_S$ of the Hilbert space $L^2(X_S)^{K_S}$ with $L^2(\R_+^*,d^*u)$, see \eqref{wS}.  With $\fourier_\mu:L^2(\R_+^*,d^*u)\to L^2(\R)$ being the Fourier transform for the multiplicative group $\R_+^*$ see \eqref{mufourier},
we obtain the unitary $\cU_S:=\F_\mu\circ  w_S$.  Finally we let $\cM_S:L^2(\R)\to L^2(\R,dm_S)$ be the unitary  given by 
\begin{equation}\label{msintro}
\cM_S(f)(s):=\left(\prod_{v\in S} L_v(\frac 12-is) \right)^{-1}	f(s),\ \ dm_S(s):=\Big\vert\prod_{v\in S} L_v(\frac 12-is)\Big\vert^2\, ds.
\end{equation}
where the $L_v$ are the local Euler factors.
We then derive the following result
\begin{theorem1}({\it Proposition} \ref{groundstate} and {\it Proposition} \ref{httransfoS})\label{thmsemilocintro}
Let $S$ be a finite set of places with $\infty\in S$, let $R_S$ be the maximal compact subring of $\prod_{p\in S\setminus \{\infty\}}\Q_p$,  and $f\in L^2(\R)^{ev}$. Then let $\eta_S(f)$ be the class of the function $1_{R_S}\otimes f$ in $L^2(X_S)^{K_S}$. 
\begin{enumerate}
\item[(i)] One has 
\begin{equation}\label{ms0intro}
	\F_\mu w_S(\eta_S(f))(s)=\Big(\prod_{p\in S\setminus\{\infty\}}L_p(\frac 12-is)\Big)(\fourier_{\mu} w_\infty f)(s).
\end{equation}
 \item[(ii)] Let the Fourier transform for $\Q_p$ be normalized so that the function $1_{\Z_p}$ is its own Fourier transform and let $\fourier_S$, acting in $L^2(X_S)$, be induced by the tensor product of the local Fourier transforms. One has $$\fourier_S\circ \eta_S=\eta_S\circ \fourierer.$$
  \item[(iii)] The unitary transformation $\cV_S:=\cM_S\circ \cU_S:L^2(X_S)^{K_S}\to L^2(\R,dm_S)$  gives the canonical form of the cyclic pair  $(D,\xi)$ where $D:= \scal$ and $\xi= \xi_S:=\eta_S(\xi_\infty )$.
\item[(iv)] The cyclic pair  $(\scal,\xi_S)$ is even and the $\Z/2\Z$ grading is given by the Fourier transform $\fourier_S$ which becomes the symmetry $s\mapsto -s$ under the unitary transformation $\cV_S$.
\end{enumerate}
\end{theorem1}

 The computation of the coefficients of the hermitian Jacobi matrix of the cyclic pair for a general $S$ as above is deferred to a forthcoming paper. In order to analyse the semilocal Sonin spaces we introduce the dual \htt transform in \cref{psodual}. In the case of the single archimedan place this transform was  introduced by J-F.~Burnol  in connection with the work of L. de Branges. The main result of this section is the construction of a hilbertian\footnote{We use the term ``hilbertian" to denote the underlying topological vector space structure of a Hilbert space} isomorphism $\theta_S:L^2(\R)^{ev}\to L^2(X_S)^{K_S}$ which connects together the semilocal Sonin spaces:

\begin{theorem2}\label{isosonintro} Let $S$ be a finite set of places with $\infty\in S$ and $\lambda >0$. Then the map $\theta_S$ induces  a hilbertian isomorphism of the Sonin spaces $\theta_S:\son_\lambda(\R,e_\infty)\to\son_\lambda(X_S,\alpha)$ where $\alpha$ is the normalized character.	
\end{theorem2}

This result suggests that the obtained invariantly defined Sonin space should, when suitably equiped with a prolate type operator, play the role of the sought for Weil cohomology. In fact we observed in \cite{CM} that, in the case of the single archimedean place, the Sonin space corresponds (up to a finite dimensional possible discrepancy) to the negative part of the spectrum of the prolate operator and this gives another constraint in the search of the semilocal analogue of the prolate operator. Moreover the results of \cite{GR} suggest a strong relation between orthogonal polynomials-which play a key role in the semilocal case- and reproducing kernels commuting with differential operators.\newline
Another conceptual point of view on the prolate operator is obtained in  \cref{metasect}. We show that the prolate operator admits (still at the formal algebraic level) a description as an element of the enveloping algebra of the Lie algebra of $SL_2(\R)$ through the  metaplectic representation of the 
double cover of $SL_2(\R)$. We compare this description of the prolate operator with \Cref{mainstart} in \cref{HT}.  We  show that the generators  of the metaplectic representation make sense in the context of cyclic pairs \ie equivalently of orthogonal polynomials.  We  prove (\Cref{dzero}) that the discrepency coefficients $d_n$  vanish precisely in the situation of  \Cref{mainstart} and their vanishing determines uniquely the cyclic pair $(\scal,\xi_\infty)$,   the associated moments and the representation.

In a forthcoming paper we shall develop the role of the Weil  representation \cite{weil} of the  metaplectic cover of the algebraic group $SL_2(\A_S)$  in the semilocal context,  thereby introducing  a second potential candidate for the semilocal prolate operator.

\section{Cyclic pairs and associated prolate operators}\label{cyclpair}
In this section we give a general construction starting from a pair $(D,\xi)$ of a selfadjoint operator $D$ acting in a Hilbert space $\cH$ and a unit vector $\xi \in \cap_{n\in \N}\operatorname{Dom}D^n$. We assume that $\cH$ is infinite dimensional and that $\xi$ is a {\em cyclic vector}, \ie that the linear span of the vectors $D^j\xi$ for $j\in \N$ is dense in $\cH$. Given such a pair we let $E_n\subset \cH$ be the linear span of the vectors $D^j\xi$ for $j\leq n$. They are finite dimensional subspaces of $\cH$ and  $\dim\, E_n=n+1$, since the vectors $D^j\xi$ are linearly independent. One has by construction $D(E_n)\subset E_{n+1}$ and the general situation is described as follows by the spectral theorem. 
\begin{thm}\label{cycpair} Let $(D,\xi)$ be a cyclic pair as above.
\begin{enumerate}
\item[(i)] There exists a probability measure $\mu$ on $\R$ and a unitary isomorphism $U: \cH\to L^2(\R, d\mu)$ which transforms $D$ in the operator of multiplication by the variable $s\in \R$ and the vector $\xi$ in the constant function equal to $1$.	
\item[(ii)] Under the isomorphism $U$ the subspace $E_n$ becomes the space of polynomials of degree $\leq n$.
\item[(iii)] The spectral measure $\int f(s)d\mu(s):=\langle f(D)\xi\vert \xi\rangle$ is a complete invariant of the cyclic pair.
\end{enumerate}
\end{thm}
\Cref{cycpair} gives the canonical form of a cyclic pair $(D,\xi)$. 
One can then apply the general theory of orthogonal polynomials and obtain an orthonormal basis $(\xi_j)$ where $\xi_0=\xi$ and $E_n$ is the span of the $\xi_j$ for $j\leq n$. One then considers the {\em number} operator $N$ which is diagonal in the orthonormal basis $(\xi_j)$ and such that $N\xi_j:=j \xi_j$.
\begin{defn} \label{prolop} Let $(D,\xi)$ be a cyclic pair. The {\em formal prolate} operator $\omega(D,\xi,\lambda)$ is the operator 
	$$\omega(D,\xi,\lambda) :=-D^2 + \lambda^2 N. $$
\end{defn}
This definition is formal inasmuch as it does not give precisely the domain of the operator. We shall see in  \cref{pso} that the standard prolate differential operator ${\bf W}_{\lambda}$ is a special case of \cref{prolop}. It is clear in that case that the obtained formal operator
is symmetric and finding the relevant selfadjoint extension is delicate.
\begin{defn} \label{z2} Let $(D,\xi)$ be a cyclic pair. A $\Z/2$-grading $\gamma$ is a unitary of square $1$, $\gamma^2=1$, $\gamma=\gamma^*$ such that $\gamma D=-D\gamma$ and $\gamma \xi=\xi$. A cyclic pair is {\em even} if it admits a $\Z/2$-grading.	
\end{defn}
\begin{prop}\label{z2grad} A cyclic pair is {\em even} iff the spectral measure $\mu$ of \Cref{cycpair} is invariant under $s\mapsto -s$. The $\Z/2$-grading is unique (if it exists) and is equal to $\gamma=\exp(i\pi N)$.	
\end{prop}
\proof Assume that the measure $\mu$ of \Cref{cycpair} is invariant under $s\mapsto -s$. Let $\gamma(f)(s):=f(-s)$ for all $f\in L^2(\R, d\mu)$. One has $\gamma^2=1$ and $\gamma$ is unitary by invariance of the measure. By construction $\gamma$ anticommutes with the multiplication by $s$ and one has $\gamma \xi=\xi$ where $\xi(s)=1$, $\forall s$. Using the unitary isomorphism $U: \cH\to L^2(\R, d\mu)$ of \Cref{cycpair} one obtains a $\Z/2$-grading for $(D,\xi)$. Conversely let $\gamma$ be a $\Z/2$-grading for $(D,\xi)$. Then the spectral measure is invariant under $s\mapsto -s$, since one has 
$$
\int h(s)d\mu(s):=\langle h(D)\xi\mid \xi \rangle =\langle h(D)\gamma\xi\mid \gamma\xi \rangle
=\langle h(-D)\xi\mid \xi \rangle=\int h(-s)d\mu(s)
$$
Let us show the uniqueness of the $\Z/2$-grading. Let $\gamma'$ be another $\Z/2$-grading. Then the product $u=\gamma\gamma'$ commutes with $D$ and fulfills $u\xi=\xi$. Thus one has $uD^j\xi=D^j\xi$ for all $j\in \N$ and since $\xi$ is cyclic one obtains $u=1$ and hence $\gamma'=\gamma$. Let us show that $\gamma=\exp(i\pi N)$. One has $\gamma D^j\xi=(-1)^jD^j\xi$. This shows that the subspaces $E_n$ (span of the $D^j\xi$ for $j\leq n$) are globally invariant under $\gamma$ and that  they decompose as the eigenspaces $E_n^{\pm}$ generated respectively by the $D^j\xi$ for $j\leq n$, $j$ even and $j$ odd. The Gram--Schmidt orthonormalization process thus takes place independently in  $E_n^{\pm}$ yielding the orthonormal basis $\xi_j$ where $\xi_j\in E_n^{\pm}$ according to the parity of $j$. It follows that $\gamma=\exp(i\pi N)$.\endproof 
\subsection{Spectral triple of a cyclic pair}
Let $(D,\xi)$ be an even cyclic pair and $\gamma$ the $\Z/2$-grading.  \cref{z2grad} suggests  to consider the even spectral triple $(\cA,\cH,D)$ where the $C^*$-algebra $\cA:=c_0(\N)$ of sequences vanishing at $\infty$ acts as follows in the Hilbert space $\cH$ of the cyclic pair:
\begin{equation}\label{spectrip}
c_0(\N)\times \cH \ni (f, \eta) \mapsto f(N)\eta \in \cH
\end{equation}
\begin{prop}\label{spectr} Let $(D,\xi)$ be a cyclic pair, and $\cA:=c_0(\N)$ act in $\cH$ by \eqref{spectrip}. Then $(\cA,\cH,D)$ is a spectral triple and the commutators $[D,f]$ make sense and are bounded for any $f\in c_c(\N)\subset c_0(\N)$. 
\end{prop}
\proof Let $f\in c_c(\N)$, it is by construction a finite linear combination of the delta functions $\delta_j(k):=0$ for $k\neq j$ and $\delta_j(j)=1$. The action \eqref{spectrip} of $\delta_j$ in $\cH$ is the rank one operator $\vert \xi_j\rangle\langle \xi_j\vert$, and its commutator with $D$ makes sense since $\xi_j\in {\rm Dom} D$ and it is equal to the bounded operator $\vert D\xi_j\rangle\langle \xi_j\vert-\vert \xi_j\rangle\langle D\xi_j\vert$. Thus the commutators $[D,f]$ are finite rank operators for any $f\in c_c(\N)\subset c_0(\N)$.\endproof 
One has $D\xi_j\in E_{j+1}$ for all $j$ and $\langle D\xi_n\vert \xi_k\rangle=0$ for $k<n-1$ using $D=D^*$. Thus one has 
\begin{equation}\label{dxij}
D\xi_n=a_{n-1}\xi_{n-1}+a_{n}\xi_{n+1}, \ \ a_n\neq 0, 
\end{equation}
where one can fix by induction the phase of the orthonormal vectors $\xi_n$ in such a way that the non-zero scalars $a_n:=\langle D\xi_n\vert \xi_{n+1}\rangle$ are positive. Let $f\in c_c(\N)$, one has $f=\sum f(j)\delta_j$ and 
$$
[D,f]=\sum f(j)[D,\delta_j]=\sum f(j)a_j\left( \vert \xi_{j+1}\rangle\langle \xi_j\vert
-\vert \xi_{j}\rangle\langle \xi_{j+1}\vert\right)+\sum f(j)a_{j-1}\left( \vert \xi_{j-1}\rangle\langle \xi_j\vert
-\vert \xi_{j}\rangle\langle \xi_{j-1}\vert\right)
$$
$$
=\sum (f(j)-f(j+1))a_j\left( \vert \xi_{j+1}\rangle\langle \xi_j\vert
-\vert \xi_{j}\rangle\langle \xi_{j+1}\vert\right)=\sum \alpha_j\left( \vert \xi_{j+1}\rangle\langle \xi_j\vert
-\vert \xi_{j}\rangle\langle \xi_{j+1}\vert\right)
$$
where $\alpha_j:=(f(j)-f(j+1))a_j$. Let $T(j):=\left( \vert \xi_{j+1}\rangle\langle \xi_j\vert
-\vert \xi_{j}\rangle\langle \xi_{j+1}\vert\right)$. One has 
$$
\Vert \sum \alpha_j T(j)\Vert\leq \Vert \sum \alpha_{2j} T(2j)\Vert+\Vert \sum \alpha_{2j+1} T(2j+1)\Vert=\max \vert \alpha_{2j}\vert + \max \vert \alpha_{2j+1}\vert
$$
since the matrices of the terms appearing in the rhs are block diagonal. Moreover one has 
$$
\langle \xi_{k+1}\vert T(j)\xi_k  \rangle =\langle \xi_{k+1}\vert \xi_{j+1}\rangle\langle \xi_j\vert\xi_k  \rangle
-\langle \xi_{k+1}\vert \xi_{j}\rangle\langle \xi_{j+1}\vert\xi_k  \rangle=\delta(j-k)
$$
so that 
$$
\langle \xi_{k+1}\vert \left(\sum \alpha_j T(j)\right)\xi_k  \rangle =\alpha_k
$$
and one obtains the inequalities 
$$
\max \vert \alpha_{j}\vert\leq \Vert \sum \alpha_j T(j)\Vert\leq \max \vert \alpha_{2j}\vert + \max \vert \alpha_{2j+1}\vert
$$
which gives 
$$
\max \vert (f(j)-f(j+1))a_j\vert\leq\Vert [D,f]\Vert\leq \max_{j\,{\rm even}} \vert (f(j)-f(j+1))a_j\vert+\max_{j\,{\rm odd}} \vert (f(j)-f(j+1))a_j\vert
$$
and 
\begin{equation}\label{dist0}
		\max \vert (f(j)-f(j+1))a_j\vert\leq\Vert [D,f]\Vert\leq 2\max \vert (f(j)-f(j+1))a_j\vert
	\end{equation}\endproof 
	 We now compute  the spectral distance function on $\N$ associated to the spectral triple  $(\cA,\cH,D)$. 
\begin{prop}\label{dist} Let $\phi(n):=\sum_{0\leq j\leq n}a_j^{-1}\in \R_+$. The spectral distance function of the spectral triple $(\cA,\cH,D)$ of \cref{spectr} is equivalent to the distance $d(n,m):=\vert \phi(n)-\phi(m)\vert$, more precisely
	\begin{equation}\label{dist1}
		\frac 12 d(n,m)\leq \sup \{\vert f(n)-f(m)\vert\mid \Vert [D,f]\Vert\leq 1\}\leq d(n,m)
	\end{equation}
\end{prop}
\proof By construction one has $a_j>0$ so that the map $\phi$ is strictly increasing and $d(n,m):=\vert \phi(n)-\phi(m)\vert$ defines a distance function on $\N$. The distance $d(n,m)$ is given by 
$$
d(n,m)=\sup \{\vert f(n)-f(m)\vert\mid \max \vert (f(j)-f(j+1))a_j\vert\leq 1\}
$$
and thus \eqref{dist1} follows from \eqref{dist0}. \endproof 
\subsection{Jacobi matrix of the prolate operator of a cyclic pair} \label{an} At the formal level, \ie ignoring the issue of selfadjointness of $D$, an even cyclic pair is determined by the sequence $(a_n)$ which gives the Jacobi matrix of the operator $D$ in the basis $(\xi_n)$ by \eqref{dxij} while no longer assuming that the $a_n$ are real.
$$
D=\left(
\begin{array}{cccccc}
 0 & a_0 & 0 & 0 & 0 & \ldots \\
 \overline a_0 & 0 & a_1 & 0 & 0 & \ldots \\
 0 &   \overline a_1 & 0 & a_2 & 0 & \ldots \\
 0 & 0 &  \overline a_2 & 0 & a_3 & \ldots \\
 0 & 0 & 0 &  \overline a_3 & 0 & \ldots\\
 \ldots & \ldots & \ldots & \ldots & \ldots & \ldots \\
\end{array}
\right)
$$
 One can then determine a  Jacobi matrix representation for the formal prolate operator of \Cref{prolop}. This operator  commutes with the grading $\gamma$ and thus is the direct sum of its restriction to the eigenspaces $\cH^\pm:=\{\xi\in \cH\mid \gamma\xi=\pm \xi\}$. These subspaces are generated by the vectors $\xi_n$ where the parity of $n$ is fixed. 
	Indeed the square of the matrix of $D$ is of the form (for $n=5$)
	$$
	\left(
\begin{array}{cccccc}
 a_0 \overline a_0 & 0 & a_0 a_1 & 0 & 0 & 0 \\
 0 & a_0 \overline a_0+a_1 \overline a_1 & 0 & a_1 a_2 & 0 & 0 \\
\overline a_0 \overline a_1 & 0 & a_1 \overline a_1+a_2 \overline a_2 & 0 & a_2 a_3 & 0 \\
 0 & \overline a_1 \overline a_2 & 0 & a_2 \overline a_2+a_3 \overline a_3 & 0 & a_3 a_4 \\
 0 & 0 & \overline a_2 \overline a_3 & 0 & a_3 \overline a_3+a_4 \overline a_4 & 0 \\
 0 & 0 & 0 &\overline  a_3 \overline a_4 & 0 & \ldots \\
\end{array}
\right)
$$
and its restrictions to $\cH^\pm$ are  Jacobi matrices of the general form
$$
\left(
\begin{array}{cccc}
 a_0 \overline a_0 & a_0 a_1 & 0 & 0 \\
 \overline a_0 \overline a_1  &a_1 \overline a_1+a_2 \overline a_2  & a_2 a_3 & 0 \\
 0 & \overline a_2 \overline a_3 &  a_3 \overline a_3+a_4 \overline a_4 & \ldots \\
 0 & 0 & \overline a_4 \overline a_5 &\ldots \\
\end{array}
\right), \ \ \left(
\begin{array}{cccc}
 a_0 \overline a_0+a_1 \overline a_1 & a_1 a_2 & 0 & 0 \\
\overline a_1 \overline a_2 &  a_2 \overline a_2+a_3 \overline a_3 & a_3 a_4 & 0 \\
 0 & \overline  a_3 \overline a_4 & a_4 \overline a_4+a_5 \overline a_5 & \ldots \\
 0 & 0 & \overline  a_5 \overline a_6 & \ldots\\
\end{array}
\right)
$$ 
We now give the matrix description of the formal prolate operator of \Cref{prolop}.
\begin{prop}\label{jacobprol} The formal prolate operator of \Cref{prolop} restricted to $\cH^\pm$ is represented by the Jacobi matrices $J^\pm$ whose matrix coefficients are $0$ except those  given by 
	\begin{equation}\label{jacobprol1}
		J^+_{n,n+1}=\overline{J^+_{n+1,n}}=-a_{2n}a_{2n+1}, \ \ J^+_{n,n}=-a_{2n}\overline a_{2n} -a_{2n-1}\overline a_{2n-1}+2 n\lambda^2
	\end{equation}
	\begin{equation}\label{jacobprol2}
		J^-_{n,n+1}=\overline{J^-_{n+1,n}}=-a_{2n+1}a_{2n+2}, \ \ J^-_{n,n}=-a_{2n}\overline a_{2n} -a_{2n+1}\overline a_{2n+1}+ (2n+1)\lambda^2
	\end{equation}
\end{prop}
\proof This follows from   \Cref{prolop} and the computation of the matrix for the restriction of $D^2$ to $\cH^\pm$.\endproof

\section{\htt transform: archimedean place}\label{pso}
This section contains a detailed proof of the following theorem.
\begin{thm}\label{mainstart} Let  $h_{n}$ be the normalized Hermite functions.
\begin{enumerate}
\item[(i)] One has $\cU(h_{0})=2^{-\frac 34}\pi^{-\frac12}L_\infty(\frac 12-is)$ where $L_\infty(z):=\pi^{-z/2}\Gamma(z/2)$ is the Euler factor at the archimedean place.
	\item[(ii)] The functions $\cU(h_{2n})$  are of the form $\cU(h_{2n})(s)=(-1)^n\cP_n(s)\cU(h_{0})(s)$ where the $\cP_n(s)$ are polynomials of the same parity as $n$ and which are orthogonal (and normalized) for the probability measure $dm(s):=(2\pi)^{-\frac 32}\vert \Gamma(\frac 14+ \frac{is}{2})\vert^2ds$ on $\R$. 
	\item[(iii)] The map $\cM:L^2(\R)\to L^2(\R,dm)$, $\cM(f)(s):=\cU(h_{0})^{-1}(s)f(s)$ is a unitary isomorphism. 
	
	\item[(iv)] The  unitary isomorphism $\cM\circ\cU:L^2(\R)^{ev}\to L^2(\R,dm)$ transforms the Hermite functions $h_{2n}$ into the orthogonal polynomials $(-1)^n\cP_n$ and the operator $\scal$ into the multiplication by $s$. 
	
	\item[(v)] Let $N$ be the grading operator  $N(\cP_n):=n\ \cP_n$  in $L^2(\R,dm)$. One has
\begin{equation}\label{wlambdconj}
(\cM\circ\cU)\circ {\bf W}_{\lambda}\circ (\cM\circ\cU)^*=-s^2+2\pi\lambda^2(4N+1)-\frac 14
\end{equation}
	 where $s^2$ is the operator of multiplication by the square of the variable in $L^2(\R,dm)$.  
\end{enumerate}	
\end{thm}
By construction the Hermite functions $h_n$ are themselves of the form $h_n(x)=H_n(x)h_0(x)$ where the $H_n$ are orthogonal (and normalized) polynomials. Since the Fourier transform $\fourier_\mu$ does not preserve the product of functions nor their polynomial nature, \Cref{mainstart} is by no means a consequence of the construction of the Hermite functions. The operator $\cU\circ S\circ \cU^*$ is the multiplication by $s$ while the Hermite operator gives the grading. The main result is that, due to the functional equation $\Gamma(x+1)=x\Gamma(x)$ of the 
$\Gamma$-function one has natural orthogonal polynomials governing the functions $\cU(h_{2n})$ where  the  $h_{n}$ are the normalized Hermite functions. 
\newpage
\subsection{Notation} 
\subsubsection{$\fourierer:L^2(\R)^{ev}\to L^2(\R)^{ev}$} The Fourier transform $\fourierer:L^2(\R)\to L^2(\R)$ 
\begin{equation}\label{Fourier}\fourierer(f)(y):=\int f(x)\exp(-2\pi i xy)dx 
\end{equation}
 restricts to the subspace $L^2(\R)^{ev}\subset L^2(\R)$ of even functions.
\subsubsection{$w_\infty:L^2(\R)^{ev}\to L^2(\R_+^*,d^*u)$} The Haar measure of the multiplicative group $\R_+^*$ of positive real numbers is $d^*u:=du/u$. The map $w_\infty$ is defined by 
\begin{equation}\label{winfty}
w_\infty(f)(u):=u^{1/2}\,f(u) \qqq u\in \R_+^*.
\end{equation}
\subsubsection{$\fourier_\mu:L^2(\R_+^*,d^*u)\to L^2(\R)$} The Fourier transform for the multiplicative group $\R_+^*$ is 
\begin{equation}\label{mufourier}
\fourier_\mu(f)(s):=\int_0^\infty f(u)\,u^{-is}d^*u.
\end{equation}
\subsubsection{$\cU: L^2(\R)^{ev}\to L^2(\R)$} The unitary transformation $ \cU:=\pi^{-\frac12}\ \fourier_\mu\circ w_\infty$ is given by
\begin{equation}\label{uu}
\cU: L^2(\R)^{ev}\to L^2(\R),\ \ \cU(f)(s)=\pi^{-\frac12}\int_0^{\infty} f(v) v^{\frac 12-is}d^*v\,.
\end{equation}
 
\subsection{Canonical form of the scaling operator}
We let $H:=x\partial_x$ and $\scal:= -i(H+\frac 12)$ acting in the Hilbert space $ L^2(\R)^{ev}$ as unbounded operators. By construction $\scal$ is the selfadjoint generator of the scaling unitary transformations 
$$
\left(\exp(it\scal)f\right)(x)=\lambda^{1/2}\,f(\lambda\,x), \ \ \lambda=e^{t/2}\qqq t\in \R.
$$
We let $\xi_\infty=h_0\in L^2(\R)^{ev}$ be the vector of norm $1$ given by 
$$
h_0(x):=2^{\frac 14}\,\exp(-\pi x^2)\qqq x\in \R.
$$
One has 
$$
2^{-\frac 14}\pi^{\frac12}\,\cU(h_{0})(s)=\int_0^{\infty}v^{1/2-is} e^{-\pi  v^2}d^*v=\frac 12 \pi ^{-\frac 14 +i\frac s 2} \Gamma \left(\frac 14 -i\frac s 2\right)
$$

We define the unitary transformation  
\begin{equation}\label{mm}
	\cM:L^2(\R)\to L^2(\R,dm), \ \ \cM(f)(s):=\cU(h_{0})^{-1}(s)f(s)
\end{equation}
where the measure $dm$ is $\vert \cU(h_{0})\vert^2 ds$
\begin{prop}\label{httransfo}\begin{enumerate}
\item[(i)] The unitary transformation $\cV:=\cM\circ \cU:L^2(\R)^{ev}\to L^2(\R,dm)$  gives the canonical form of the cyclic pair  $(\scal,\xi_\infty))$ where $\scal:= -i(H+\frac 12)$ and $\xi_\infty= h_0$.
\item[(ii)] The  cyclic pair  $(\scal,\xi_\infty)$ is even  and the (unique) grading $\gamma$ is equal to the Fourier transform $\fourierer:L^2(\R)^{ev}\to L^2(\R)^{ev}$.
\end{enumerate}
\end{prop}
\proof $(i)$~The conjugate of $\scal$ by $\cU$ is the multiplication by the variable $s$ since 
$$
\int_0^{\infty} ((H+\frac 12)f)(v) v^{\frac 12-is}d^*v =\int_0^{\infty} \partial_v(v^{\frac 12}f(v)) v^{-is}dv= is   \int_0^{\infty} f(v) v^{\frac 12-is}d^*v
$$
This remains true when conjugating by $\cV$, while by construction $\cV(\xi_\infty)$ is the constant function $\cV(\xi_\infty)(s)=1$.\newline
$(ii)$~The vector $\xi_\infty= h_{0}$ is invariant under $\fourierer$ and this transformation is unitary of square $1$ and anticommutes with $\scal$, thus the result follows from \cref{z2grad}.
\endproof 
\subsection{Number operator $N$}\label{numberop}
We relate the number operator $N$ of the cyclic pair $(\scal,\xi_\infty)$ with the Hermite differential operator ${\bf H}=-\partial^2+(2 \pi  q)^2$.
\begin{prop}\label{numberN} 
\begin{enumerate}	\item[(i)] The number operator $N$ of the cyclic pair $(\scal,\xi_\infty)$  is given by 
\begin{equation}\label{herm}
	N=\frac{{\bf H}}{8\pi}-\frac 14, \ \ \ \ {\bf H}=-\partial^2+(2 \pi  q)^2
\end{equation}
\item[(ii)] The eigenfunctions of $N$ are the even Hermite functions $h_{2n}\in L^2(\R)^{ev}$ where
\begin{equation}\label{hermfunctions}
	h_{2n}(x)=\sum_0^{n} (-1)^{n-k} 2^{-n+3k+\frac 14}\frac{ ((2n)!)^{\frac12} }{(2k)! (n- k)!} \pi^{ k}x^{2k}e^{-\pi  x^2}
\end{equation}	
\end{enumerate}
\end{prop}
\proof The Hermite functions \eqref{hermfunctions} are eigenfunctions for the operator ${\bf H}$  and give the orthonormal basis  of eigenfunctions for the restriction of ${\bf H}$ to $L^2(\R)^{ev}$ and the eigenvalues are $2 \pi (4n+1)$. The subspace $E_n\subset L^2(\R)^{ev}$  linear span of the vectors $\scal^j\xi_\infty$ for $j\leq n$ is formed of the even functions $e^{-\pi  x^2}P(x)$ where $P(x)$ is an even polynomial of degree $\leq 2n$. It coincides with the linear span of the functions $h_{2j}$, $0\leq j\leq n$. Thus one obtains $(i)$ and $(ii)$.\endproof
As a corollary we get equation \eqref{wlambdconj} of  \Cref{mainstart}.
\begin{cor}
\begin{equation}\label{wlambdconj1}
(\cM\circ\cU)\circ {\bf W}_{\lambda}\circ (\cM\circ\cU)^*=-s^2+2\pi\lambda^2(4N+1)-\frac 14
\end{equation}	
\end{cor}
\proof One has by \eqref{WLambdaq1}
$$
({\bf W}_{\lambda}\psi)(q) = \,\partial (q^2 \partial \,
\psi(q)) + \lambda^2 \,{\bf H}\,\psi(q) =\left(-\scal^2 + \lambda^2 \,{\bf H}-\frac 14\right)\psi(q)
$$ 
which gives \eqref{wlambdconj1} using \eqref{herm} since $\cV\circ \scal\circ \cV^*$ is given by multiplication by $s$ by \cref{httransfo}.\endproof 
 \subsection{ Jacobi matrix for $\scal$}\label{sectj1}
 This section should be compared, for the formulas of the tridiagonal matrix, with \cite[Section 8]{Osipov}  (formulas based on Hermite series).

We let
\begin{equation}\label{basicpols}
\cP_m(x):=\sqrt{(2 m)!}\sum_0^m (-1)^k 2^{3 k-m}\frac{ \prod _{j=0}^{k-1} \left(j-\frac{i x}{2}+\frac{1}{4}\right)}{(2 k)! (m-k)!}
\end{equation}

\begin{prop}\label{fourierfmuall}
\begin{enumerate}
\item[(i)] The polynomials $\cP_m$ are orthonormal polynomials with respect to the probability  measure 	$(2 \pi)^{-3/2}\vert \Gamma \left(\frac 14 -i\frac {x}{ 2}\right)\vert ^2 dx$ on the  line.
\item[(ii)] The matrix acting on column vectors\footnote{The matrix with entries $\langle x\cP_n,\cP_m\rangle$ is the transposed matrix}  of the operator of multiplication by $x$ in the basis of the $\cP_m$ is the hermitian Jacobi matrix with $0$ on the diagonal and $\alpha_n=\frac{i}{2} \sqrt{(2 n+1) (2 n+2)}$ in the upper  diagonal.
\end{enumerate}
\end{prop}
\proof $(i)$~The transform $\cV(h_{2m})$ by the unitary $\cV:=\cM\circ \cU:L^2(\R)^{ev}\to L^2(\R,dm)$ is given  by the polynomial $(-1)^m\cP_m$ as one gets using the equality
\begin{equation}\label{gammafunct}
\int_0^{\infty}x^{1/2-is} \pi^{k}x^{2k}e^{-\pi  x^2}d^*x=\frac 12 \pi ^{-\frac 14 +i\frac s 2} \Gamma \left(\frac 14 -i\frac s 2+k\right)
\end{equation}
and the functional equation for $\Gamma(z)$. The polynomial $\cP_0=1$ gives the normalization. The orthonormality then follows from unitarity of $\cV$.\newline
 $(ii)$~Using $\fourier_\mu\circ w$ the operator $H+\frac 12$ corresponds to the multiplication by $ix$, thus it is enough to determine the  coefficient of $x^{2n+2}e^{-\pi  x^2}$ in $(H+\frac 12)h_{2n}$. The term of highest degree in  $h_{2n}$ is 
$$
  2^{2n+\frac 14}(2n)!^{-\frac12}  \pi^{n}x^{2n}e^{-\pi  x^2}
 $$
 When applying the operator $H= x \partial_x$ to $e^{-\pi  x^2} x^k$ one gets 
 $$
H(e^{-\pi  x^2} x^k)= e^{-\pi  x^2} x^k \left(-2 \pi  x^2 +k\right)
 $$
 Thus the coefficient of $x^{2n+2}e^{-\pi  x^2}$ in $(H+\frac 12)h_{2n}$ is $\left(-2 \pi  \right)2^{2n+\frac 14}(2n)!^{-\frac12}  \pi^{n}$ while in $h_{2n+2}$ it is $2^{2n+2+\frac 14}(2n+2)!^{-\frac12}  \pi^{n+1}$. Thus for this highest term one has $(H+\frac 12)h_{2n}\sim c_n h_{2n+2}$ for 
 $$
 c_n=\left(-2 \pi  \right)2^{2n+\frac 14}(2n)!^{-\frac12}  \pi^{n}/\left(2^{2n+2+\frac 14}(2n+2)!^{-\frac12}  \pi^{n+1}\right)=-\frac{1}{2} \sqrt{(2 n+1) (2 n+2)}
 $$
 which gives, using the orthonormal property of the $h_n$, the equality 
 \begin{equation}\label{basiceq}
 \langle (H+\frac 12)h_{2n},h_{2n+2}\rangle=-\frac{1}{2} \sqrt{(2 n+1) (2 n+2)}
\end{equation}
 To check the signs one takes the simplest case, one has
 $$
 ((H+\frac 12)h_{0})(x)=(2^{\frac 14-1}-2^{\frac 14+1}\pi x^2)e^{-\pi  x^2}, \ \ h_2(x)=\frac{4 \pi  x^2-1}{\sqrt[4]{2}}
 $$
 so that $(H+\frac 12)h_{0}=-\frac{1}{\sqrt{2}}h_2$. For $m=1$ one has 
 $
 \cP_1(x)=i \sqrt{2}\ x
 $. 
 Thus $x\cP_0=-\frac{i}{\sqrt{2}}\cP_1$. From \eqref{basiceq} we thus obtain that the matrix of the multiplication by $x$, acting on column vectors is of the form
$$
\left(
\begin{array}{cccccc}
 0 & \frac{i}{\sqrt{2}} & 0 & 0 & 0 & 0 \\
 -\frac{i}{\sqrt{2}} & 0 & i \sqrt{3} & 0 & 0 & 0 \\
 0 & -i \sqrt{3} & 0 & i \sqrt{\frac{15}{2}} & 0 & 0 \\
 0 & 0 & -i \sqrt{\frac{15}{2}} & 0 & i \sqrt{14} & 0 \\
 0 & 0 & 0 & -i \sqrt{14} & 0 & 3 i \sqrt{\frac{5}{2}} \\
 0 & 0 & 0 & 0 & -3 i \sqrt{\frac{5}{2}} & 0 \\
\end{array}
\right)
$$
\ie that one has for all $n$ 
 \begin{equation}\label{basiceq1}
 x\cP_n=\overline{\alpha_n}\,\cP_{n+1}+\alpha_{n-1}\cP_{n-1}, \ \ \   \alpha_n=\frac{i}{2} \sqrt{(2 n+1) (2 n+2)}
\end{equation}
so that the component of $x\cP_n$ on $\cP_{n+1}$ is $-\frac{i}{2} \sqrt{(2 n+1) (2 n+2)}$.
\endproof 
We show in  \cref{metric1} below that for the cyclic pair corresponding to the archimedean place, the spectral distance is equivalent to the metric induced on $\N^*\subset \R_+^*$ by the invariant metric of the Lie group $\R_+^*$.
\begin{prop}\label{metric1} The spectral metric for the cyclic pair $(\scal,\xi_\infty)$ of  \cref{httransfo} is equivalent to the metric induced on $\N^*\subset \R_+^*$ by the invariant metric of the Lie group $\R_+^*$.	
\end{prop}
\proof 
The sequence $(a_n)$ of \eqref{dxij}, obtained after a change of phase  in the orthonormal basis, is given by $a_n=\sqrt{(n+\frac 12)(n+1)}$. For $n>0$ one has 
$$
1\leq a_n \, \log(1+\frac 1n)\leq \sqrt 3 \log 2
$$
which shows that the metric induced on $\N^*\subset \R_+^*$ by the invariant metric of the Lie group $\R_+^*$ is equivalent to the metric $d(n,m):=\vert \phi(n)-\phi(m)\vert$ of  \cref{dist}. In turns this latter metric is equivalent to the spectral metric by  \cref{dist}.\endproof 
\subsection{Jacobi matrix of the prolate operator} \label{jacofan}
We now determine the two Jacobi matrices given by the restrictions ${\bf{W}}_\lambda^\pm$ of the prolate operator \eqref{wlambda} to the even and odd subspaces $\cH^\pm$ as in \cref{jacobprol}.  
In fact one could compute directly these Jacobi matrices in the basis of Hermite functions  and compare with \cite[Section 8]{Osipov}  (formulas based on Hermite series). \newline
The Jacobi matrix ${\bf{W}}_\lambda^+$ is of the following form
$$
\left(
\begin{array}{ccccc}
 2 \pi  \lambda ^2-\frac{3}{4} & \sqrt{\frac{3}{2}} & 0 & 0 & 0 \\
 \sqrt{\frac{3}{2}} & 18 \pi  \lambda ^2-\frac{43}{4} & \sqrt{105} & 0 & 0 \\
 0 & \sqrt{105} & 34 \pi  \lambda ^2-\frac{147}{4} & 3 \sqrt{\frac{165}{2}} & 0 \\
 0 & 0 & 3 \sqrt{\frac{165}{2}} & 50 \pi  \lambda ^2-\frac{315}{4} & \sqrt{2730} \\
 0 & 0 & 0 & \sqrt{2730} & 66 \pi  \lambda ^2-\frac{547}{4} \\
\end{array}
\right)
$$
For a Jacobi matrix $J$, we use the notation $a_n$, $n\geq 0$ for the coefficients $J_{n,n+1}=\overline{J_{n+1,n}}$, and $b_n$, $n\geq 0$ for the diagonal, \ie $b_n=J_{n,n}$.
\begin{prop}\label{matrixjj}	
\label{explicitjac}
The  coefficients of the Jacobi matrix of the prolate operator ${\bf{W}}_\lambda^+$ are given by 
\begin{equation}\label{explicitjac1}
a_n=\frac 14\left((4 n+1)(4n+2) (4 n+3)(4n+4)\right)^{1/2},\ \ 
 b_n=-8 n^2-2 n-\frac{3}{4}+2 \pi  \lambda ^2(8n+1) \end{equation}
In the odd case, \ie for ${\bf{W}}_\lambda^-$, they are 
\begin{equation}\label{explicitjac2}
a_n=\frac 14\left((4 n+3)(4n+4) (4 n+5)(4n+6)\right)^{1/2}, \ \ 
b_n=-8 n^2-10 n-\frac{15}{4}+2 \pi  \lambda ^2(8n+5) \end{equation}
\end{prop}
\proof This follows from \eqref{wlambdconj1} combined with  \cref{jacobprol}  applied to the coefficients 
$\alpha_n=\frac{i}{2} \sqrt{(2 n+1) (2 n+2)}$ determined in  \cref{fourierfmuall}.\endproof 
\subsection{Link with the map $\cE$ and zeta}\label{sectmapE}
In \cite{CCVJ} we exhibited minuscule eigenvalues of the quadratic form associated to the Weil explicit formulas restricted to test functions  supported in a fixed compact interval. The corresponding eigenfunctions were constructed using the image through the map $\cE$ of the positive spectrum of the prolate operator ${\bf W}_{\lambda}$.  We used these functions to condition the canonical spectral triple of the circle  in such a way that they belong to the kernel of the perturbed Dirac operator and obtained the low lying  zeros of the Riemann zeta function. 
For $\lambda\to \infty$ the  eigenfunctions  for positive eigenvalues of  ${\bf W}_{\lambda}$  are approximated by the eigenfunctions of  the operator ${\bf H}=-\partial^2+(2 \pi  q)^2$. The latter eigenfunctions are the Hermite functions $\{h_{2n}\}$.\newline  Following   \cite[Proposition 2.24]{CMbook} we now explain in which sense the conditioning of the scaling operator by  $\cE(\{h_{2n}\})$, performed by taking a quotient (in place of an orthogonal complement as above) relates to the zeros of zeta.

Let $\cS(\R)_0^{\rm ev}$ be the subspace of the even part of the Schwartz space $\cS(\R)$ obtained by imposing the two conditions $f(0)=\widehat f(0)=0$. We describe simple linear combinations of the functions $h_{2m}$ which belong to $\cS(\R)_0^{\rm ev}$ and then compute their image under the map $\fourier_\mu\circ\cE$, where 
$$
\cE:=w_\infty \circ \Sigma, \ \ \Sigma(f)(x):=\sum_\N f(nx).
$$
The Fourier transform $\fourier_{e_\R}(h_{2m})$ is  $(-1)^m h_{2m}$ and thus the two conditions $f(0)=\widehat f(0)=0$ are fulfilled by the two families of functions 
\begin{equation}\label{evenoddfamily}
\psi^+_\ell:=h_{4\ell}-\frac{h_{4\ell}(0)}{h_0(0)}h_0, \ \ \psi^-_\ell:=-h_{4\ell+2}+\frac{h_{4\ell+2}(0)}{h_2(0)} h_2
\end{equation}
One has $
h_{2n}(0)=(-1)^n \frac{ 2^{\frac{1}{4}-n} \sqrt{(2 n)!}}{n!}
 $. Thus one gets 
\begin{equation}\label{psileven}
\psi^+_\ell(x)=\sum_1^{2\ell} (-1)^k 2^{-2\ell+3k+\frac 14}\frac{ ((4\ell)!)^{\frac12} }{(2\ell- k)! (2k)!} \pi^{k}x^{2k}e^{-\pi  x^2}
\end{equation}
In the odd case one has $h_2(x)/h_2(0)=e^{-\pi  x^2} \left(4 \pi  x^2-1\right)$.  The coefficient of $\pi x^2 e^{-\pi  x^2}((4\ell+2)!)^{\frac12}$ in $h_{4\ell+2}-\frac{h_{4\ell+2}(0)}{h_2(0)} h_2$ is 
$$
 2^{3 -2 \ell-1+\frac 14}\frac{ 1}{(2\ell)! 2!}-4\times 2^{-2 \ell-1+\frac 14}\frac{ 1}{(2\ell+1)! }=2^{-2 \ell+\frac 14}\frac{ 4\ell }{(2\ell+1)! }
$$ 
thus one has \begin{small}
\begin{equation}\label{psilodd}
\psi^-_\ell(x)=\left(-\frac{  \ell\, 2^{-2 \ell+\frac 94} ((4\ell+2)!)^{\frac12}}{(2 \ell+1)!}\pi x^2+\sum_2^{2\ell+1} (-1)^{k} 2^{3 k-2 \ell-1+\frac 14} \frac{((4\ell+2)!)^{\frac12}}{ (2 \ell+1-k)!(2 k)!} \pi^{k}x^{2k}\right)e^{-\pi  x^2}
\end{equation}
\end{small}
We now relate the Fourier transform   $\fourier_\mu(w_\infty(\psi_\ell^\pm))$ to the polynomials \eqref{basicpols}.
 \begin{lem}\label{fourierfmu}
 \begin{enumerate} \item[(i)] The Fourier transform   $\fourier_\mu(w_\infty(\psi_\ell^\pm))$  is equal to 
the product of the archimedean local factor $\pi ^{-\frac{z}{2}} \Gamma \left(\frac{z}{2}\right)$ at $z=\frac 12 -is$,  by the following polynomials 
\begin{equation}\label{basicpolyeven}
P^+_\ell(s):=\sum_1^{2\ell} (-1)^k 2^{-2\ell+3k-\frac 34}\frac{ ((4\ell)!)^{\frac12} }{(2\ell- k)! (2k)!}\prod_0^{k-1}  \left(j-\frac{i\, s}{2}+\frac{1}{4}\right)
\end{equation}
which is even, divisible by $\frac 14+s^2$, and with real coefficients and
$$
P^-_\ell(s):=-\frac{  \ell \,2^{-2 \ell+\frac 54} ((4\ell+2)!)^{\frac12}}{(2 \ell+1)!}\left(-\frac{i s}{2}+\frac{1}{4}\right)+
$$
\begin{equation}\label{basicpolysodd}
+\sum_2^{2\ell+1} (-1)^{k} 2^{3 k-2 \ell-\frac 74} \frac{((4\ell+2)!)^{\frac12}}{ (2 \ell+1-k)!(2 k)!} \prod_0^{k-1}  \left(j-\frac{i s}{2}+\frac{1}{4}\right)
\end{equation}
which is odd, divisible by $\frac 14+s^2$, and with purely imaginary coefficients.
\item[(ii)] One has 
\begin{equation}\label{basicpolrelate}
P^+_\ell(s)=2^{-\frac 34}\left(\cP_{2\ell }(s)-\cP_{2\ell }(\frac i2)\right), \ \ P^-_\ell(s)=2^{-\frac 34}\left( \cP_{2\ell +1}(s)+2i\, s\,\cP_{2\ell+1 }(\frac i2)\right)
\end{equation}
	\end{enumerate}
\end{lem}
\proof $(i)$~The formulas \eqref{basicpolyeven} and \eqref{basicpolysodd} follow directly from \eqref{psileven} and \eqref{psilodd} together with \eqref{gammafunct}. The  divisibility by $\frac 14+s^2$ follows from the equality
$$
\fourier_\mu(w(\psi_\ell^\pm))(\frac i2)=\int_0^\infty x\,\psi_\ell^\pm(x)d^*x=\frac 12 \int_{-\infty}^\infty \psi_\ell^\pm(x)dx=0
$$
It follows from \eqref{psileven} and \eqref{psilodd} that $P^\pm_\ell(i/2)=0$ and the parity of the polynomial $P^\pm_\ell$ then shows that it is divisible by $\frac 14+s^2$. \newline
$(ii)$~This follows since \eqref{basicpolrelate} uniquely specifies the required correction of the rescaled polynomials $\cP_m$ to make them divisible by $\frac 14+s^2$ without altering their parity.\endproof 
 The Riemann-Landau $\Xi$ function
\begin{equation}\label{xifunct}
\Xi(s)=\frac{z(z-1)}{2} \Gamma(z/2) \pi^{-z/2} \zeta(z)\,,\quad
z=\frac 12 + is
\end{equation}
 is entire, of Hadamard order one, even and real-valued for real $s$, and one has
 $$
\Xi(s)=\Xi(0)\,  \prod \left( 1-\frac{s^2}{\alpha^2}\right), \ \ \Xi(0)=-\frac{\zeta \left(\frac{1}{2}\right) \Gamma \left(\frac{1}{4}\right)}{8 \sqrt[4]{\pi }}\sim 0.497121
$$ 
where $\frac 12+i\alpha$ runs through the  zeros of zeta with  positive imaginary part.\newline
Let then $P^\pm_\ell(s)=-\frac 12 (\frac 14+s^2)R^\pm_\ell(s)$ define the polynomials $R^\pm_\ell(s)$.
The next proposition shows that these polynomials give the factorization of $\fourier_\mu(\cE(\psi^\pm_\ell))$ as a multiple of the $\Xi$ function, and that all polynomial multiples of $\Xi$ are obtained in this way.
We let $\cH_{\leq 1}$ be the Hadamard topological ring of entire functions of order $\leq 1$.
\begin{prop}\label{prope}\begin{enumerate} \item[(i)] One has the equality 
$$
\fourier_\mu(\cE(\psi^\pm_\ell))(s)=R^\pm_\ell(s)\, \Xi(s)
$$
\item[(ii)] The spectrum of the operator of multiplication by $s$ in the quotient of $\cH_{\leq 1}$  by the  closure of the subspace spanned by the  $\fourier_\mu(\cE(\psi^\pm_\ell))$ is the set of $s\in \C$ such that $\frac 12 +is$ is a non-trivial zero of zeta.
\end{enumerate}
\end{prop}
\proof This follows from Lemma \ref{fourierfmu}, as in   \cite[Proposition 2.24]{CMbook}. 
\endproof

\section{Semilocal case}\label{sectsemilocal}
 \subsection{Notations}\label{notat}
 Let $S$ be a finite set of places, $\infty\in S$, and $\A_{S}$ 
the locally compact ring
\begin{equation}\label{AQS}
 \A_{S}=\prod_{v\in S} \Q_v.
\end{equation}
This ring contains $\Q$ as a subring using the diagonal embedding. Let
$\Q_S$ denote the subring of $\Q$ given by rational numbers whose
denominator only involves primes $p \in S$. In other words,
\begin{equation}\label{AQSring}
 \Q_S=\{q\in \Q\,|\, |q|_v\leq 1\,,\ \forall v\notin S\}\,.
\end{equation}
The group $\Gamma:=\Q^*_S$ of invertible elements of the ring $\Q_S$ is of
the form
\begin{equation}\label{GL1QS}
\Gamma=\GL_1(\Q_S)= \{ \pm p_1^{n_1} \cdots p_k^{n_k} \, :\,  p_j
\in S \setminus\{ \infty \} \,,\, n_j\in \Z\}.
\end{equation}
 The semilocal ad\`ele  class space  $X_{S}$ is by definition the quotient
\begin{equation}\label{XQS}
X_{S}:=\A_{S}/\Gamma,
\end{equation}
and  we let $\pi_S:\A_{S}\to X_S$ be the canonical projection. 
The module extends to a multiplicative map $\vert \bullet \vert_S$ from the ring $\A_S=\prod_{v\in S} \Q_v$ to $\R_+$
\begin{equation}\label{module}
\vert (u_v)_{v\in S} \vert_S=\prod \vert u\vert_v   \in \R_+
\end{equation}
and by construction this map passes to the quotient $X_S=\A_{S}/\Gamma$. 
The groups
\begin{equation}\label{GL1AS}
\GL_1(\A_{S})= \prod_{p\in S} \GL_1(\Q_p), \ \ C_{S}=\GL_1(\A_{S})/\Gamma
\end{equation}
act naturally by multiplication on the quotient $X_{S}$ and the orbit of
$1\in \A_{S}$  gives an embedding
$
C_{S}\to X_{S}.
$
The complement of $C_S$ in $X_S$ is of measure zero for the product of the  Haar measures of the additive groups of the local fields (which is preserved by the action of the countable group $\Gamma$). Using the Radon-Nikodym derivative of the Haar measures of the multiplicative groups with respect to the Haar measure of the additive groups, one obtains a unitary identification
\begin{equation}\label{wS}
 w_S :L^2(X_{S})\to L^2(C_{S})
 \end{equation} (see \cite{CMbook} Proposition 2.30).
 We also recall (see  \cite[Eqs. (2.223) and (2.239)]{CMbook} that  $C_{S}$  is a
 modulated locally compact group with module 
\begin{equation}\label{Modulus}
{\rm Mod}_S(\lambda) =   |\lambda|_S :=  \prod_{p\in S} |\lambda_p |, \quad
 \forall \lambda = (\lambda_p) \in  C_{S} 
 \end{equation}
which is (non-canonically) isomorphic to \,$\R^*_+ \times K_S $ ,
where $K_S$ is the kernel of  ${\rm Mod}_S$.
 \subsection{Semilocal \htt transform}
 We extend the formulas \eqref{uu} and \eqref{mm} to the semilocal case. We use the  notation of \cref{notat}.\newline
 Let $S$ be a finite set of places, $\infty\in S$, let $R_S$ be the maximal compact subring of $\prod_{p\neq \infty}\Q_p$. For $f\in L^2(\R)^{ev}$ let $\eta_S(f)$ be the class of the function $1_{R_S}\otimes f$ in $L^2(X_S)$. More precisely one has, ignoring the summation over $\pm 1$ since $f$ is even, and letting $\Gamma_+:=\Gamma\cap \Q_+$,
 \begin{equation}\label{etas}
 \eta_S(f)(x):=\sum _{\Gamma_+}\,f(\gamma \,\tilde x)\qqq \tilde x \mid \pi_S(\tilde x)=x.
 \end{equation}
  The restriction of $\eta_S(f)$ to the fundamental domain $F=R_S^*\times \R_+^*$ of the action of $\Gamma$ on $\GL_1(\A_{S})= \prod_{p\in S} \GL_1(\Q_p)$ is $R_S^*$-invariant and, since $1_{R_S}\otimes f$ vanishes on $\gamma (1\times u)$ unless $\gamma\in \Z$.  One has 
 $$
 \eta_S(f)(1\times u)=\sum_{\Gamma_+\cap\, \Z}\,f(\gamma \,u).
 $$
 It follows that 
\begin{equation}\label{essl0}
 w_S(\eta_S(f))(u)=\cE_S(f)(u)
   \end{equation}
 where
\begin{equation}\label{essl}
 \cE_S(f)(u):=u^{1/2}\sum_{\Gamma_+\cap\, \Z}\,f(\gamma \,u)
  \end{equation}
 
 \begin{prop}\label{groundstate} Let $f\in L^2(\R)^{ev}$. Then with $L_p(z)=(1-p^{-z})^{-1}$ for $z\in \C$, \begin{enumerate}
\item[(i)] One has 
\begin{equation}\label{ms0}
	\F_\mu (\cE_S(f))(s)=\Big(\prod_{p\in S\setminus\{\infty\}}L_p(\frac 12-is)\Big)(\fourier_{\mu} w_\infty f)(s).
\end{equation}
\item[(ii)] Let $\lambda>0$. One has $\eta_S(P_\lambda)\subset P_\lambda^S$ where $P_\lambda^S$ is the  subspace 
 of $L^2(X_S)^{K_S}$ of functions with support in $\{x\mid \vert x\vert \leq \lambda\}$. 
 \item[(iii)] Let the Fourier transform for $\Q_p$ be normalized so that the function $1_{\Z_p}$ is its own Fourier transform and let $\fourier_S$, acting in $L^2(X_S)$, be induced by the tensor product of the local Fourier transforms. One has $$\fourier_S\circ \eta_S=\eta_S\circ \fourierer.\vspace{-0.05in}$$
 \item[(iv)] One has $\eta_S(\widehat{ P_\lambda})\subset \widehat{P_\lambda^S}$ where $\widehat{ P_\lambda}$ and $\widehat{P_\lambda^S}$ are the images of $P_\lambda,P_\lambda^S$ by the Fourier transform. 
 \end{enumerate}
\end{prop}

\begin{proof} $(i)$~For simplicity we consider the case of $S=\{p,\infty\}$. Then  one has
\[
\fourier_\mu\mathcal E_p(f)(s)=L_p(\frac 12-is)(\fourier_{\mu} w_\infty f)(s).
\]
Indeed, one has 
\begin{align} \fourier_\mu\mathcal E_p(f)(s)=\int_0^\infty u^{1/2}\sum f(p^ku)u^{-is} d^*u=\sum_{k=0}^\infty \int_0^\infty f(p^k u)u^{\frac 12-is} d^*u =\sum_k\int_0^\infty f(v) p^{-\frac k2} p^{iks} v^{\frac 12-is} d^* v\nonumber\\
=\sum_k p^{-\frac k 2}p^{iks}\int_0^\infty f(v) v^{\frac 12 -is} d^*v=\int_0^\infty f(v) v^{\frac 12 -is} d^*v\,(1-p^{-\frac 1 2-is})^{-1}=L_p(\frac 12-is)(\fourier_{\mu} w_\infty f)(s).\nonumber\end{align}
$(ii)$~The module $\vert x\vert$ is equal to $1$ on principal ideles and hence on elements $\gamma\in \Gamma=\{\pm\prod p_v^{n_v}\mid n_v\in \Z\}$. Thus $\vert x\vert$ makes sense on $X_S=\left(\prod \Q_v\right)/\Gamma$ and so does $P_\lambda^S$. Let $f\in P_\lambda$ and $\eta_S(f)$ be the class of the function $1_{R_S}\otimes f$ in $L^2(X_S)$. Let $x=(y,u)\in R_S\times \R$ with $\vert x\vert>\lambda$. One has  $$
\lambda <\vert x\vert=\vert y\vert \vert u\vert\leq \vert u\vert \Rightarrow \vert u\vert>\lambda \Rightarrow  f(u)=0.
$$
$(iii)$~This follows since $1_{R_S}$ is its own Fourier transform. \newline
$(iv)$~Follows from $(ii)$ and $(iii)$.
\end{proof}
We let $\cU_S:=\F_\mu\circ  w_S$ and $\cM_S:L^2(\R)\to L^2(\R,dm_S)$ the unitary  given by 
\begin{equation}\label{ms}
\cM_S(f)(s):=\left(\prod_{v\in S} L_v(\frac 12-is) \right)^{-1}	f(s),\ \ dm_S(s):=\vert\prod_{v\in S} L_v(\frac 12-is)\vert^2\, ds
\end{equation}
We let $K_S$ be the kernel of the module $C_{S}\to \R_+^*$ and denote by $\scal$ the generator of the scaling action of $\GL_1(\R)_+\subset \GL_1(\A_{S})= \prod_{p\in S} \GL_1(\Q_p)$ on $L^2(X_S)^{K_S}$.
\begin{prop}\label{httransfoS} 
\begin{enumerate} \item[(i)] The unitary transformation $\cV_S:=\cM_S\circ \cU_S:L^2(X_S)^{K_S}\to L^2(\R,dm_S)$  gives the canonical form of the cyclic pair  $(D,\xi)$ where $D:= \scal$ and $\xi= \xi_S:=\eta_S(\xi_\infty )$.
\item[(ii)] The cyclic pair  $(\scal,\xi_S)$ is even and the grading is given by the Fourier transform $\fourier_S$ which becomes the symmetry $s\mapsto -s$ under the unitary transformation $\cV_S$.
\end{enumerate}
\end{prop}
\proof $(i)$~The unitary $\cU_S:L^2(X_S)^{K_S}\to L^2(\R)$ is an isomorphism and transforms the operator $\scal$ into the multiplication by the variable $s$ as in the proof of  \cref{httransfo}. By \eqref{ms0} applied to $\xi_\infty$ one has 
$$
\cU_S(\xi_S)(s)=\cU_S(\eta_S( \xi_\infty))(s)=\Big(\prod_{p\in S\setminus\{\infty\}}L_p(\frac 12-is)\Big)(\fourier_{\mu} w_\infty( \xi_\infty))(s)=\prod_{v\in S} L_v(\frac 12-is) 
$$
$$
\cV_S(\xi_S)(s)=\cM_S\circ \cU_S(h_S)(s)=1.
$$
Thus $\cV_S:L^2(X_S)^{K_S}\to L^2(\R,dm_S)$ is an isomorphism which transforms the operator $\scal$ into the multiplication by the variable $s$ and the vector $\xi_S$ into the constant function $1$.\newline
$(ii)$~One checks directly that the Fourier transform $\fourier_S$ anticommutes with $\scal$. It fixes the vector $\xi$ by  \cref{groundstate}. The uniqueness of the grading shows that the Fourier transform $\fourier_S$  becomes the symmetry $s\mapsto -s$ under the unitary transformation $\cV_S$.\endproof 
We let $\iota_S:L^2(\R,dm)\to L^2(\R,dm_S)$ be the identity map $\iota_S(f)(s):=f(s)$, $\forall s\in \R$.
\begin{prop}\label{tauS}
\begin{enumerate} \item[(i)] The map $\iota_S$ is bounded with bounded inverse.
\item[(ii)] One has the following commutative diagram:
\begin{equation}
\begin{gathered} 
\xymatrix@C=25pt@R=25pt{ 
L^2(\R)^{ev} \ar[d]^{\cV} \ar[r]^{\eta_S} &
L^2(X_S)^{K_S}\ar[d]^{\cV_S}  \\
 L^2(\R,dm) \ar[r]^{\iota_S} & L^2(\R,dm_S)
 } \label{srole}
\end{gathered}
\end{equation}
\end{enumerate}
\end{prop}
\proof 
 $(i)$~This follows since the function $\prod_{p\in S\setminus\{\infty\}}L_p(\frac 12-is)$ is bounded with bounded inverse, so that the Radon-Nikodym derivatives $dm/dm_S$ and $dm_S/dm$ are both bounded.\newline
 $(ii)$~This follows from \eqref{ms0}.\endproof 
 
 \subsection{Semilocal Hermite operator}
 Let $S$ be a finite collection of places (including the archimedean one) and $(\scal,\xi_S)$ the cyclic pair of  \cref{httransfoS}.
As in \cref{cyclpair}, the associated Hermite operator $N_S$ is defined as the grading operator associated with the  filtration $(E_n)$ of the Hilbert space $L^2(X_S)^{K_S}$ by the subspaces generated by the iterates $\scal^j\xi$, $0\leq j\leq n$ 
\begin{thm}\label{hermsemiloc}
\begin{enumerate}
\item[1.] For $S=\{\infty\}$ the semilocal Hermite operator $N_S$ is the restriction of the harmonic oscillator to  $L^2(\R)^{ev}$.
\item[2.] The eigenfunctions of the semilocal Hermite operator $N_S$ are elements of $L^2(X_S)^{K_S}$ of the form $\eta_S (P_n^S(x)e^{-\pi x^2})$, where $P_n^S$ are polynomials obtained by orthonormalization and induction.
\item[3.] The matrix of $\scal$ in the above orthonormal basis of $L^2(X_S)^{K_S}$, is a Jacobi hermitian matrix.
\end{enumerate}
\end{thm}
\begin{proof}
1.~See \cref{numberN}.\newline
2.~One has $\xi_S=2^{1/4}\,\eta_S( e^{-\pi x^2})$ and the operator $\scal= -i(H+\frac 12)$ acts as scaling on the archimedean variable. The subspaces $E_n^S\subset L^2(X_S)^{K_S}$ are obtained by iteration of the operator $\scal$ applied to the function $\eta_S( e^{-\pi x^2})$ and the scaling operator commutes with $\eta_S$. Thus $E_n^S$ is the image by $\eta_S$ of the space of products  $P(x)e^{-\pi x^2}$ where $P$ is an even polynomial of degree $\leq 2n$. Since $\eta_S$ is not unitary the orthogonalisation process delivers polynomials $P_n^S$ which depend upon $S$.  \newline
3. Follows from \eqref{dxij}, \ie the general theory of orthogonal polynomials.
\end{proof}

Thus we have reached a candidate for the analogue of the Hermite operator in the semilocal case and also for the prolate operator as in  \Cref{prolop}:
\[
\mathbf W_{\lambda, S} = (H+\frac 12)^2 + \lambda^2 N_S.
\]
\begin{rem}\label{care} \begin{enumerate}
\item[(i)] Due to the non-unitary nature of $\eta_S$ one has $N_S\circ \eta_S\neq \eta_S\circ N_\infty$ unless $S=\{\infty\}$ but the filtrations associated to the subspaces  $N\leq n$ correspond to each other by the map $\eta_S$ as shown in  \Cref{hermsemiloc} $(ii)$.
\item[(ii)] Let $\vert\bullet \vert_S:
X_S=\A_{S}/\Gamma\to \R_+$ be the module as in \eqref{module}. Another guess for the semilocal analogue of the Hermite operator is to express the latter when $S=\{\infty\}$ in terms of a sum of the multiplication by $\vert x\vert_S^2$ and its conjugate under the Fourier transform. 
\item[(iii)] Note that unless $S=\{\infty\}$ one has 
\begin{equation}\label{care1}
\vert \bullet\vert_S^2\, \eta_S( f)\neq \eta_S(\vert \bullet\vert^2f).
\end{equation}	
Indeed, taking $S=\{p,\infty\}$ for simplicity, one has for $f$ an even function 
$$
w_S\eta_S( f)=\cE_p(f), \ \ \cE_p(f)(u)=u^{1/2}\sum_\N f(p^n u)
$$
 the image of the l.h.s. of \eqref{care1} under the map $w_S$ is the function $u^2\cE_p(f)(u)$,   while the image of the r.h.s. of \eqref{care1} under the map $w_S$ is $\cE_p(\vert \bullet\vert^2f)$ which evaluated at $u$ gives the different expression
 $$
 u^{1/2}\sum_\N u^2 p^{2n}f(p^n u)
 $$
 \end{enumerate}
\end{rem}

\subsection{The dual \htt transform}\label{psodual}
We let as above $S$ be a finite set of places containing $\infty$, and  $\cU_S:=\F_\mu\circ  w_S$. We let \begin{equation}\label{ess}E_S(s):=\prod_{p\in S}L_p(\frac 12+is)\end{equation}
 and consider the Hilbert space $L^2\Big(\R, \frac{ds}{|E_S(s)|^2}\Big)$. We let $\beta_S:L^2(\R)\to L^2\Big(\R, \frac{ds}{|E_S((s)|^2}\Big)$ be the unitary  given by 
\begin{equation}\label{msdual}
\beta_S(f)(s):=\left(\prod_{v\in S} L_v(\frac 12+is) \right)	f(s)
\end{equation}
\begin{defn} \label{dualht} The dual \htt transform is the unitary 
\begin{equation}\label{dualht0}
\upsilon_S: L^2(X_S)^{K_S}\to L^2\Big(\R, \frac{ds}{|E_S(s)|^2}\Big), \ \ \upsilon_S = \beta_S\circ \cU_S
\end{equation}
\end{defn}
\begin{prop}\label{pairing}
\begin{enumerate} \item[(i)] The following equality defines a sesquilinear pairing $\langle 
	L^2\Big(\R, \frac{ds}{|E_S(s)|^2}\Big)\vert L^2(\R,dm_S)\rangle$
	\begin{equation}\label{pairingdef}
\langle \xi \vert \eta\rangle_\R:=\int \xi(x)\overline{\eta(x)}dx\qqq \xi \in L^2\Big(\R, \frac{ds}{|E_S(s)|^2}\Big),\ \eta \in L^2(\R,dm_S)\end{equation}
\item[(ii)] For any $\xi,\eta\in L^2(X_S)^{K_S}$ one has $\langle \upsilon_S\xi \vert \cV_S\eta\rangle_\R=\langle \xi\vert \eta\rangle$.
\end{enumerate}
\end{prop}
\proof The proof of $(i)$ and $(ii)$ follows from the equality valid for any place $v$ and $s\in \R$,
$$
\overline{L_v(\frac 12+is)}=L_v(\frac 12-is).
$$
and combining \eqref{ms} with \eqref{msdual}.\endproof

  \subsection{The Sonin space in the local case}
   The local definition of Sonin's space is
 \begin{defn}\label{soninp} Let $\K$ be a local field and $\alpha$ an additive character of $\K$. Let $\lambda>0$. The Sonin space $\son_\lambda(\K,\alpha)$ is the  subspace of the $L^2$-space of square integrable functions on $\K$ defined as follows \[
\son_\lambda(\K,\alpha):= \{ f \in L^2(\K)\mid f(x)=0 \ \& \ \fourier_\alpha  f(x)=0\quad \forall x, \vert x\vert <\lambda\}
\]
where $\F_\alpha$ denotes the Fourier transform with respect to $\alpha$.
\end{defn}
We use the imaginary exponential $e:\Q/\Z\to \{z\in \C\mid \vert z\vert =1\}$, $e(x):=\exp(2\pi i x)$.
We endow $\Q_p$ with the additive character (obtained using the embedding $\Q_p/\Z_p\subset \Q/\Z $)  
\begin{equation}\label{ep}
  e_p: \Q_p\to \Q_p/\Z_p\subset \Q/\Z \stackrel{e}{\to} \{z\in \C\mid \vert z\vert =1\}\subset \C
\end{equation}
which is equal to $1$ on the maximal compact subring $\Z_p\subset \Q_p$.\newline
Given a function $f \in L^2(\K)$ and $a\in \K^*$, we let $f_a$ be the function $f_a(x):=f(ax)$. The Fourier transform fulfills the equality 
$$
\fourier (f_a)=\frac{1}{\vert a\vert}\,\fourier(f)_{a^{-1}}
$$

\begin{prop} \label{soninppp} Let $p$ be a finite prime and  $\K=\Q_p$, $\alpha=e_p$. The $\Z_p^*$-invariant part of $\son_1(\Q_p,e_p)$ is one dimensional, with generator $\sigma_p:=\epsilon_0-\frac 1p \epsilon_1$ where $\epsilon_n$ is the characteristic function of $\{ x\in \Q_p\mid \vert x\vert =p^n\}$. Moreover one has $\fep \sigma_p=\sigma_p$.	
\end{prop}
\proof The function $\epsilon_n$ is equal to $(\zp)_{p^n}-(\zp)_{p^{n-1}}$ since one has
$$
\epsilon_n(x)=\zp(p^n x)-\zp(p^{n-1}x).
$$
Moreover one has $\zp=\sum_0^\infty \epsilon_{-\ell}$ and $(\epsilon_\ell)_{p^k}=\epsilon_{\ell+k}$. 
The Fourier transform of the function $\epsilon_n$ is then  given by
\begin{equation}\label{fourierpp}
\fep({\epsilon}_n)=p^n (\zp)_{p^{-n}}-p^{n-1}(\zp)_{p^{-n+1}}=p^n(1-\frac 1p)\sum_0^\infty \epsilon_{-n-k}-p^n\frac 1p \epsilon_{-n+1}.
\end{equation}  By \eqref{fourierpp} one has
$$
\fep \sigma_p=\fep \epsilon_0- \frac 1p \fep \epsilon_1=\left((1-\frac 1p)\sum_0^\infty \epsilon_{-k}-\frac 1p \epsilon_{1}\right) -\left((1-\frac 1p)\sum_0^\infty \epsilon_{-1-k}-\frac 1p \epsilon_{0}\right) =\epsilon_0-\frac 1p \epsilon_1
$$
thus $\fep \sigma_p=\sigma_p$.  By construction one has $\sigma_p(x)=0$ $\forall x, \vert x\vert <1$ thus   $\sigma_p \in \son_1(\Q_p,e_p)$. \newline
Let us now show that any $\psi \in \son_1(\Q_p,e_p)$ which is  $\Z_p^*$-invariant is proportional to $\sigma_p$. Since $\psi$ is  $\Z_p^*$-invariant, there exists coefficients $a_n\in \C$ such that $\psi=\sum_{n\in \Z}a_n \epsilon_n$ and the $L^2$ condition means that $\sum p^n\vert a_n\vert^2<\infty$. Since $\psi(x)=0$ when $\vert x\vert<1$ all the coefficients $a_n$ for $n<0$ vanish and one has $\psi=\sum_{n\geq 0}a_n \epsilon_n$.  We now compute the Fourier transform $\fep \psi$, which by \eqref{fourierpp} is 
$$
\fep \psi=\sum_{n\geq 0} a_n\left( p^n(1-\frac 1p)\sum_{k=0}^\infty \epsilon_{-n-k}-p^n\frac 1p \epsilon_{-n+1}\right).
$$
The sum on the right only involves $\epsilon_{j}$ for $j<0$ except for the terms where $n=0$ and $k=0$ in the sum $\sum_0^\infty \epsilon_{-n-k}$ and the term with $n=0$ and $n=1$ in $-p^n\frac 1p \epsilon_{-n+1}$. Thus the terms not involving $\epsilon_{j}$ for $j<0$  are given by the expression $a_0(1-\frac 1p)\epsilon_0-a_1\epsilon_0-a_0\frac 1p \epsilon_1$. Since  $\psi \in \son_1(\Q_p,e_p)$  the sum of the terms involving the $\epsilon_{j}$ for $j<0$ is equal to $0$ and thus one gets the equality 
$$
\fep \psi=a_0(1-\frac 1p)\epsilon_0-a_1\epsilon_0-a_0\frac 1p \epsilon_1=\left(a_0(1-\frac 1p)-a_1\right)\epsilon_0-a_0\frac 1p \epsilon_1=a\epsilon_0+b\epsilon_1.
$$
Finally in order that $\fep(a\epsilon_0+b\epsilon_1)$ vanishes for $\vert x\vert<1$ requires that it vanishes for $x=0$ \ie that $\int (a\epsilon_0(x)+b\epsilon_1(x))dx=0$ which means that $a+pb=0$. This gives $\left(a_0(1-\frac 1p)-a_1\right)-a_0=0$, \ie $a_0=-pa_1$, so that $\psi$ is a scalar multiple of $\sigma_p$. \endproof 
\subsection{The Sonin space in the semilocal case}\label{sectsonin}
Let $S$  be as above, and $\vert\bullet \vert_S:
X_S=\A_{S}/\Gamma\to \R_+$ be the module as in \eqref{module}. We let $\alpha$ be the character of the additive group $\A_{S}=\prod_S \Q_v$ obtained as the product of the $e_p$ (see \eqref{ep}) with $e_\infty(x):=\exp(2\pi i x)$. We let $\sigma_S:=\otimes_{S\setminus \{\infty\}} \sigma_p$.
\begin{defn}\label{soninsemiloc} Let $\lambda>0$. The semilocal Sonin space  $\son_\lambda(X_S,\alpha)$ is the  subspace of  the Hilbert space $L^2
(X_S)^{K_S}$ defined as follows
 \[
\son_\lambda(X_S,\alpha):= \{ f \in L^2
(X_S)^{K_S}\mid f(x)=0 \ \& \ \fourier_S f(x)=0\quad \forall x, \vert x\vert <\lambda\}
\]
where $\fourier_S $ denotes the Fourier transform with respect to $\alpha$.
\end{defn}
\begin{prop}\label{sonintensor}
\begin{enumerate} \item[(i)] Let $f\in \son_\lambda(\R,e_\infty)$, $\theta_S(f)$ be the class of the function  $\sigma_S\otimes f$  in $L^2
(X_S)$. Then $\theta_S(f)$ belongs to $\son_\lambda(X_S,\alpha)$.
\item[(ii)] One has
	\begin{equation}\label{tensorsonin1}
\fourier_\mu(w_S(\theta_S(f)))(s)=\fourier_\mu(w_\infty(f))(s)\times \prod_{S\setminus \{\infty\}}  \left ( 1-p^{-\frac 12 -is}\right)
\end{equation}
\end{enumerate}
\end{prop}
\proof $(i)$~Let $z=(x,y)\in R_S\times \R= \A_S$ be such that $\vert z\vert=\vert x\vert_{S\setminus \{\infty\}}\vert y\vert_\infty <\lambda$. Then either $\vert x\vert_{S\setminus \{\infty\}}<1$ or $\vert y\vert_\infty<\lambda$. In both cases one gets $\sigma_S(x)f(y)=0$ as well as $\widehat {\sigma_S}(x)\widehat {f}(y)=0$. The same holds for the elements  $\gamma z\in \A_S$, $\gamma \in \Gamma_+$ so that the class of $\sigma_S\otimes f$  in $L^2
(X_S)$ belongs to $\son_\lambda(X_S,\alpha)$.\newline
$(ii)$~For simplicity of notation we assume  $S=\{p,\infty\}$. One has, as in \eqref{etas}, 
$$
w_S(\sigma_S\otimes f)(1\times\lambda):= \lambda^{1/2}\,\sum_{n\in \Z} (\sigma_p\otimes f)(p^n,p^n \lambda)=\lambda^{1/2}\left(f(\lambda)-\frac 1p f(\lambda/p) \right)
$$
since $\sigma_p=\epsilon_0-\frac 1p \epsilon_1$.
Let $g=w_\infty(f)$ then one has  
$$
\lambda^{1/2}f(\lambda/p)=p^{1/2}(\lambda/p)^{1/2}f(\lambda/p)=p^{1/2}g(\lambda/p)
$$
so that one gets
$$
w_S(\sigma_S\otimes f)(1\times \lambda)=g(\lambda)-p^{-1/2}g(\lambda/p).
$$
Next one has 
$$
\fourier_\mu(g)(s)=\int g(\lambda)\lambda^{-is}d^*\lambda
$$
so that with $g_1(\lambda)=g(\lambda/p)$ one gets with $\lambda=pu$,
$$
\fourier_\mu(g_1)(s)=\int g_1(\lambda)\lambda^{-is}d^*\lambda=\int g(\lambda/p)\lambda^{-is}d^*\lambda=\int g(u)(pu)^{-is}d^*u=p^{-is}\fourier_\mu(g)(s)
$$
which gives the required equality \eqref{tensorsonin1}.\endproof
\subsection{The stability of  Sonin spaces}
We keep the above notation. Since $ \left ( 1-p^{-\frac 12 -is}\right)^{-1}=L_p(\frac 12 +is)$,  we rewrite \eqref{tensorsonin1} as
\begin{equation}\label{ms1}
	\F_\mu w_S(\theta_S(f))(s)=\Big(\prod_{p\in S\setminus\{\infty\}}L_p(\frac 12+is)\Big)^{-1}(\fourier_{\mu} w_\infty f)(s).
\end{equation}
\begin{prop}\label{groundstatedual} Let $\lambda>0$.
\begin{enumerate}
\item[(i)] Let  $\fourier_S$ be the Fourier transform in $L^2(X_S)$. One has $\fourier_S\circ \theta_S=\theta_S\circ \fourierer$.
 \item[(ii)] Let $\iota'_S$ be the map $f\mapsto f$, $L^2\Big(\R, \frac{ds}{|E_\infty(s)|^2}\Big)\to L^2\Big(\R, \frac{ds}{|E_S(s)|^2}\Big)$. One has the commutative diagram 
\begin{equation}
\begin{gathered}
\xymatrix@C=25pt@R=25pt{ 
L^2(\R)^{ev} \ar[d]^{\upsilon_\infty} \ar[r]^{\theta_S} &
L^2(X_S)^{K_S}\ar[d]^{\upsilon_S}  \\
 L^2\Big(\R, \frac{ds}{|E_\infty(s)|^2}\Big) \ar[r]^{\iota'_S} & L^2\Big(\R, \frac{ds}{|E_S(s)|^2}\Big)
 } \label{srole1}
 \end{gathered}
 \end{equation}
\item[(iii)] Let $f,g\in L^2(\R)^{ev}$. One has $\langle \theta_S(f)\vert \eta_S(g)\rangle=\langle f\vert g\rangle$.
\end{enumerate}
\end{prop}
\proof By  \cref{sonintensor}, one has $\theta_S(\son_\lambda(\R,e_\infty))\subset \son_\lambda(X_S,\alpha)$ where $\son_\lambda(X_S,\alpha)$ is the semilocal Sonin space. \newline$(i)$~This follows from the equality $\fep \sigma_p=\sigma_p$.	\newline
$(ii)$~Let $f\in L^2(\R)^{ev}$. We compute $\upsilon_S\circ \theta_S(f)$. By \eqref{dualht0} one has $\upsilon_S=\beta_S\circ \cU_S=\beta_S\circ\F_\mu w_S$.  By \eqref{ms1}  
$$
\upsilon_S\circ \theta_S(f)(s)=\left(\beta_S\F_\mu w_S(\theta_S(f))\right)(s)=\left(\prod_{v\in S} L_v(\frac 12+is) \right)\Big(\prod_{p\in S\setminus\{\infty\}}L_p(\frac 12+is)\Big)^{-1}(\fourier_{\mu} w_\infty f)(s)=
$$
$$
=L_\infty(\frac 12+is)(\fourier_{\mu} w_\infty f)(s)=\upsilon_\infty(f)(s).
$$
$(iii)$~The map $\cU_S=\F_\mu w_S:L^2(X_S)^{K_S}\to L^2(\R)$ is unitary so that one has 
$$
\langle \theta_S(f)\vert \eta_S(g)\rangle=\langle \cU_S\theta_S(f)\vert \cU_S\eta_S(g)\rangle
$$
and using \eqref{ms1} for the left term and \eqref{ms0} for the right term, one gets 
$$
\langle \theta_S(f)\vert \eta_S(g)\rangle=\langle\Big(\prod_{p\in S\setminus\{\infty\}}L_p(\frac 12+is)\Big)^{-1}(\fourier_{\mu} w_\infty f)\vert \Big(\prod_{p\in S\setminus\{\infty\}}L_p(\frac 12-is)\Big)(\fourier_{\mu} w_\infty g)\rangle=
\langle f\vert g\rangle
$$
which gives the required equality. \endproof 
We are now ready to prove the following fact
\begin{thm}\label{isosonin} Let $S\ni \infty$ be a finite set of places and $\lambda >0$. Then the map $\theta_S$ is a hilbertian isomorphism of the Sonin spaces $\theta_S:\son_\lambda(\R,e_\infty)\to\son_\lambda(X_S,\alpha)$ where $\alpha$ is the normalized character.	
\end{thm}
\proof We already know that $\theta_S(\son_\lambda(\R,e_\infty))\subset\son_\lambda(X_S,\alpha)$ by  \cref{sonintensor} $(i)$. To show that one has equality, let $h\in  
\son_\lambda(X_S,\alpha)$. The commutative diagram \eqref{srole1} shows that the map $\theta_S:L^2(\R)^{ev}\to L^2(X_S)^{K_S}$ is bounded with bounded inverse, hence there exists a unique $f\in L^2(\R)^{ev}$ such that $\theta_S(f)=h$. To show that $f\in \son_\lambda(\R,e_\infty)$  it is enough to show that $f$ is orthogonal to all elements of $ P_\lambda$ and $\widehat{ P_\lambda}$. By  \cref{groundstate} $(ii)$ and $(iv)$, one has $\eta_S(P_\lambda)\subset P_\lambda^S$ and $\eta_S(\widehat{ P_\lambda})\subset \widehat{P_\lambda^S}$. Thus one obtains using  \cref{groundstatedual} $(iii)$ that for any element  $g\in P_\lambda$, or $g \in \widehat{P_\lambda^S}$ one has
$$
\langle f \vert g\rangle=\langle \theta_S(f)\vert \eta_S(g)\rangle=0
$$
since by construction the subspace  $\son_\lambda(X_S,\alpha)$ is orthogonal to $P_\lambda^S$ and to $\widehat{P_\lambda^S}$. Thus we conclude that $f\in \son_\lambda(\R,e_\infty)$. \endproof 

 \subsection{Hilbertian  spaces of entire functions}
 
The stability exhibited by  \Cref{isosonin} shows that the hilbertian structure of the Sonin spaces $\son_\lambda(X_S,\alpha)$ is independent of $S$ and the theory of  Hilbert spaces of entire functions \cite{deB-book} in fact provides a canonical realization of these spaces which we now describe. It is important to point out that the choice of the finite set $S$ plays a key role in fixing the inner product in the Hilbertian space.
\subsubsection{Link with Hilbert spaces of entire functions}
The link between Sonin spaces (called in \cite{deB-book} ``Sonine spaces") and Hilbert spaces of entire functions, due to de Branges \cite{deB, deB-book}, 
is that the map  (\cf \ also \cite{burnol-1,burnol-3})
\begin{align} \label{checkM}
\son_\lambda(\R,e_\infty) \ni f \mapsto
 \check{\cM}(f) (z):=\pi ^{-\frac z 2} \Gamma \left(\frac z 2\right)\int_\R f(t) t^{1-z} d^*t \,
\end{align}
sends $ \son_\lambda(\R,e_\infty)$ onto a Hilbert space of entire functions, 
denoted here ${\cB}_\lambda$, which belongs
to the class of de Branges spaces.  
\subsubsection{Stability under semilocal amplification}
The restriction of $\check{\cM}(f) (z)$ to the critical line, \ie $z=\frac 12 +is$, coincides with the dual \htt transform in the case $S=\{\infty\}$, 
\begin{equation}\label{htteq}
	\check{\cM}(f) \left(\frac 12 +is\right)=\upsilon_\infty(f)(s)
\end{equation}
as follows from  \Cref{dualht}, \eqref{msdual}, \eqref{dualht0}.
\begin{prop}\label{propbb} Let $S\ni \infty$ be a finite set of places and $\lambda >0$. The map $\upsilon_S$ is an isomorphism of hilbertian spaces $\son_\lambda(X_S,\alpha)$ with ${\cB}_\lambda$ and one has the commutative diagram 
	\begin{equation}\label{ciccia}
\begin{gathered}
\xymatrix{
\son_\lambda(\R,e_\infty) \ar[rr]^{\theta_S}\ar[rd]_{\upsilon_\infty} &
 &
\son_\lambda(X_S,\alpha)  \ar[ld]^{\upsilon_S}\\
& {\cB}_\lambda}
\end{gathered}
\end{equation}
\end{prop}
	\proof This follows from the diagram \eqref{srole1} together with \eqref{htteq}.\endproof 
Note that  ${\cB}_\lambda$ inherits different inner products from its embedding in $L^2\Big(\R, \frac{ds}{|E_S(s)|^2}\Big)$. 

\subsubsection{De Branges space ${\cB}^S_\lambda$}
Let $z^\sharp$ designate
the symmetric of $z$ with respect to the critical line $L=\frac 12 + i \R$.
 A de Branges space is a Hilbert space $\cK$ of entire functions satisfying the following axioms:
 \begin{enumerate}
 	\item[(1)] the evaluation at any  $w\in \C$ is continuous;
 	\item[(2)] the map $F \mapsto F^{\#}$, $F^{\#}(w)=\overline{F\left(w^{\#}\right)}$ is a unitary anti-isometry of $\cK$;
 	\item[(3)] if  $F(\gamma)=0$ then $G(w)=\left(w-\gamma^{\#}\right) /(w-\gamma) F(w)$ belongs to $\cK$ and $\Vert G\Vert=\Vert F\Vert$.
 \end{enumerate}
It follows from these axioms that for any finite $S\ni \infty$ the space ${\cB}^S_\lambda$ obtained by endowing the space ${\cB}_\lambda$ with the norm induced by  its embedding in $L^2\Big(\R, \frac{ds}{|E_S(s)|^2}\Big)$ is a de Branges space. 

\section{Metaplectic representation in the Jacobi picture} \label{metasect} 
\subsection{Infinitesimal representation and prolate operator}
 We recall that the metaplectic representation $\varpi$ of $\frak{sl}(2, \R)$ on $L^2(\R)$
is realized at the infinitesimal level by letting the standard basis $\{h, e_+, e_-\}$, 
with $[h, e_+] = 2e_+$, $[h, e_-] = -2e_-$, $[e_+, e_-] = h$, act on
$\cS(\R)$, via the differential operators 
\begin{align} \label{metaplrep}
\varpi(h):=x\partial_x + \frac 12 , \quad \varpi(e_+):=i \pi x^2 , \quad \varpi(e_-):=\frac{i}{4\pi} \partial_x^2 .
\end{align}
Thus $\varpi(h)$ is up to a factor of $i$ the scaling $\scal$ while the Hermite operator, is obtained as $\varpi(k)$  using the generator $k= i(e_- - e_+)$ of the
maximal compact subgroup $K=\widetilde{SO}(2)$.  
With these notations one has 
\begin{align} \label{b-k}
 \varpi(k) =\pi x^2 -\frac{1}{4\pi} \partial_x^2
 \end{align}
is the Hermite operator with spectrum $\{ n+\frac 12 ; \, n \in \Z^+ \}$ 
and the Hermite functions 
$$
h_n=\left(2^{2n-\frac 12} n! \ \pi^n \right)^{-1/2} P_n(x) e^{-\pi x^2/2}, \quad \text{ where } \quad
  P_n= (-1)^n e^{2\pi x^2} \partial_x^n \left(e^{-2\pi x^2}\right) ,
$$
provide the associated eigenfunctions forming an orthonormal basis. \newline
Thus, switching to the other standard $\C$-basis $\{k, n_+, n_-\}$, 
where $n_\pm=\frac 12\bigl(h \mp i(e_+ + e_-)\bigr)$, the representation $\varpi$ is determined by  
\eqref{b-k} together with  
\begin{align} \label{b-n}
\varpi(n_+) &= \frac{\partial^2}{8 \pi }+\frac{x\,\partial}{2} +\frac{1}{2} \pi  x^2 +\frac{1}{4}
=\left(\frac{\partial}{2 \sqrt{2 \pi }}+x\sqrt{\frac{\pi }{2}}\right)^2 = a^2 \\
\varpi(n_-) &= - \frac{\partial^2}{8 \pi }+\frac{x\,\partial}{2} -\frac{1}{2} \pi  x^2 +\frac{1}{4}
= -\left(-\frac{\partial}{2 \sqrt{2 \pi }}+x\sqrt{\frac{\pi }{2}}\right)^2  = - (a^*)^2 
 \end{align}
where  $a=\frac{\partial}{2 \sqrt{2 \pi }}+x\sqrt{\frac{\pi }{2}}$ and $a^*=-\frac{\partial}{2 \sqrt{2 \pi }}+x\sqrt{\frac{\pi }{2}}$ are the {\em annihilation} and
 {\em creation} operators, in terms of which $\varpi(k) = aa^*+a^*a$.
They function as {\em lowering} and {\em raising} operators with respect to the
basis $\{ h_n ; \, n \in \Z^+\}$ since
\begin{align} \label{up-down} 
a (h_0) =0, \quad \  a (h_n)= \sqrt{\frac n2}\, h_{n-1} , \quad \text{and} \quad  a^* (h_n)= \sqrt{\frac{n+1}{2}}\, h_{n+1}
\quad \forall \, n\in \Z^+ .
\end{align}
The Casimir operator $C:=h^2+2(e_+e_-+e_-e_+)$ fulfills $\varpi(C)=-\frac 34$.
In particular this shows that $L^2(\R) = L^2(\R)^{ev} \oplus L^2(\R)^{odd}$ represents
the decomposition of $\varpi$ into two irreducible subrepresentations 
of lowest weight $\frac 12$ and $\frac 32$ respectively.

\begin{prop}\label{proprolate} At the formal level the operator ${\bf W}_\lambda$ is of the form $\varpi(\cW_\lambda)$ where $\cW_\lambda$ is the following element of the enveloping algebra $\cU(\frak{sl}(2, \R))$
$$
\cW_\lambda=h^2 +4 \pi \lambda^2\, k-\frac 14.
$$    
\end{prop}
\proof This follows using \eqref{WLambdaq1}, \eqref{metaplrep} and \eqref{b-k}.\endproof

 \subsection{Metaplectic representation and orthogonal polynomials} \label{HT}
We shall now present the metaplectic representation as a special case of a general construction for othogonal polynomials. We first consider, as in  \cref{an}, the general set-up of orthogonal polynomials with respect to a measure $d\mu$ on $\R$ such that $d\mu(-s)=d\mu(s)$. Modulo a change of phase in the orthonormal basis we may assume that  the coefficients $a_n$ of the   Jacobi matrix $A$ for the multiplication by the variable  $s\in \R$ are positive, 
\begin{equation}\label{matrixA}
A=\left(
\begin{array}{cccccc}
 0 & a_0 & 0 & 0 & 0 & \ldots \\
 a_0 & 0 & a_1 & 0 & 0 & \ldots \\
 0 &   a_1 & 0 & a_2 & 0 & \ldots \\
 0 & 0 &   a_2 & 0 & a_3 & \ldots \\
 0 & 0 & 0 &   a_3 & 0 & \ldots\\
 \ldots & \ldots & \ldots & \ldots & \ldots & \ldots \\
\end{array}
\right)
\end{equation}
The  coefficients $a_n$ are determined by the moments $c_n:=\int t^nd\mu(t)$ by (see \cite[2.2.7, 2.2.15]{Szego}) 
\begin{equation}\label{Szego}
a_n=\frac{k_n}{k_{n+1}}
, \ \ \ 
k_n=(D_{n-1}/D_n)^{1/2}
\end{equation}
 where the $D_n$ are the determinants of Hankel matrices of the form, since the odd moments vanish, 
$$
D_5=\left(
\begin{array}{cccccc}
 c_0 & 0 & c_2 & 0 & c_4 & 0 \\
 0 & c_2 & 0 & c_4 & 0 & c_6 \\
 c_2 & 0 & c_4 & 0 & c_6 & 0 \\
 0 & c_4 & 0 & c_6 & 0 & c_8 \\
 c_4 & 0 & c_6 & 0 & c_8 & 0 \\
 0 & c_6 & 0 & c_8 & 0 & c_{10} \\
\end{array}
\right)
$$
To investigate the  Lie algebra generated by $A$ and the grading operator $N$, $NP_n=nP_n$. We let
\begin{equation}\label{eplusminus}
F_+:=N+\frac{1}{2i}[A,N],\ \ F_-:=N-\frac{1}{2i}[A,N]	
\end{equation}
\begin{lem}\label{eplusminus1}\begin{enumerate} \item[(i)] ~The double commutator $[[A,N],N]=A$.
\item[(ii)] The commutator $[F_+,F_-]$ is equal to $-iA$.
\item[(iii)] The double commutator $[A,[A,N]]$ is equal to $f(N)$ for the function
\begin{equation}\label{fn}
	f(n)= 2 (a_{n-1}^2- a_n^2)
\end{equation} 
\item[(iv)] One has  $[A,F_+]=2i F_++\Upsilon$ where the diagonal matrix $\Upsilon$ has the entries $d_n$ with 
\begin{equation}\label{dn}
	\frac 1id_n=-2n+a_n^2-a_{n-1}^2
\end{equation}
\item[(v)] The matrix $F_+-\frac 12 i A$ is upper triangular and is the sum $N-i\mathcal S$ where $\mathcal S$ is the weighted shift
$$
\mathcal S=\left(
\begin{array}{cccccc}
 0 & a_0 & 0 & 0 & 0 & 0 \\
 0 & 0 & a_1 & 0 & 0 & 0 \\
 0 & 0 & 0 & a_2 & 0 & 0 \\
 0 & 0 & 0 & 0 & a_3 & 0 \\
 0 & 0 & 0 & 0 & 0 & a_4 \\
 0 & 0 & 0 & 0 & 0 & 0 \\
\end{array}
\right)
$$	
\end{enumerate}
\end{lem}
\proof $(i)$~The commutator $[A,N]$ is of the form 
\begin{equation}\label{comm}
\left(
\begin{array}{cccccc}
 0 & a_0 & 0 & 0 & 0 & 0 \\
 -a_0 & 0 & a_1 & 0 & 0 & 0 \\
 0 & -a_1 & 0 & a_2 & 0 & 0 \\
 0 & 0 & -a_2 & 0 & a_3 & 0 \\
 0 & 0 & 0 & -a_3 & 0 & a_4 \\
 0 & 0 & 0 & 0 & -a_4 & 0 \\
\end{array}
\right)
\end{equation}
and the double commutator $[[A,N],N]=A$.\newline
$(ii)$~One has 
$$
[F_+,F_-]=[N+\frac{1}{2i}[A,N],N-\frac{1}{2i}[A,N]]=\frac{1}{i}[[A,N],N]=-i A.
$$
 $(iii)$~This important fact comes from computing $[A,[A,N]]$ which gives this matrix
$$
\left(
\begin{array}{cccccc}
 -2 a_0^2 & 0 & 0 & 0 & 0 & 0 \\
 0 & 2 a_0^2-2 a_1^2 & 0 & 0 & 0 & 0 \\
 0 & 0 & 2 a_1^2-2 a_2^2 & 0 & 0 & 0 \\
 0 & 0 & 0 & 2 a_2^2-2 a_3^2 & 0 & 0 \\
 0 & 0 & 0 & 0 & 2 a_3^2-2 a_4^2 & 0 \\
 0 & 0 & 0 & 0 & 0 & \ldots \\
\end{array}
\right)
$$
$(iv)$~One has 
$$
[A,F_+]=[A,N]+\frac{1}{2i}[A,[A,N]]=2i F_++ \Upsilon, \ \ d_n=-2in+\frac{1}{2i}f(n).
$$
$(v)$~By \eqref{comm}, the matrix $F_+-\frac 12 i A$ is upper triangular.\endproof 
 We look for a representation $\sigma$ of $\frak{sl}(2, \R)$ such that, with  $A$ as in \eqref{matrixA} and $k= i(e_- - e_+)$, 
 \begin{align} \label{h}
 \sigma(h)=-iA, \quad \sigma(k)=N+c\end{align}
 for some constant $c$. 
The identity
 \begin{align} \label{hhk}
 [h, [h,k]]=-2i([h,e_+]+[h,e_-])=-4i(e_+ - e_-)=4k
 \end{align}
  translates into the relation
 \begin{align} \label{aak}
  [A, [A, \sigma(k))]] = -4\sigma(k) .
 \end{align} 
 By  \cref{eplusminus1} (iii), $[A, [A, N]]$ is the diagonal matrix with diagonal entries  
 $
  2\big(a_{n-1}^2 - a_n^2\big)  
$. 
The 
identity \eqref{aak} is thus equivalent to  
\begin{equation}\label{diffsqan}
 2 (a_{n-1}^2- a_n^2)=-4(n +c), \ \ \forall n\geq 0.  
\end{equation}
One has $[h,k]=-2i(e_++e_-)$ so that 
\begin{equation}\label{eplusminusricu}
e_+=\frac i2\left(k+\frac 12 [h,k]\right), \ \ e_-=\frac i2\left(-k+\frac 12 [h,k]\right)  
\end{equation}
Thus with the notation \eqref{eplusminus} one gets, using \eqref{h} 
$$
\sigma(e_+)=\frac i 2(F_++c), \ \ \sigma(e_-)=\frac i 2(-F_--c)
$$
Thus by  \Cref{eplusminus1} $(iv)$  one obtains 
$$
[\sigma(h),\sigma(e_+)]=\frac 12 [A,F_+]=i F_++\frac 12\Upsilon.
$$ The Lie algebra equality $[h,e_+]=2 e_+$ together with \eqref{dn} and \eqref{diffsqan} then  determine uniquely $c=\frac 14$. This uniquely determines the $a_n$, the representation 
$\sigma$, and  $E_\pm:=-i \sigma(e_\pm)$ 
 are given by
\begin{align} \label{cap-E}
E_+ =  \frac 1 2 \left(N+ \frac 14\right)+\frac 14 [A,N] , \qquad 
E_- = -\frac 1 2 \left(N+ \frac 14\right) + \frac 14 [A,N],
\end{align}

\begin{thm}\label{dzero} 
\begin{enumerate}
\item[(i)] The positive coefficients $a_n$ are specified uniquely  in order for \eqref{h} to define a representation of the Lie algebra $\frak{sl}(2, \R)$, one has  
\begin{equation}\label{coefffixed}
a_n=\frac{1}{2} \sqrt{(2 n+1) (2 n+2)}.
\end{equation}
\item[(ii)] The Hamburger moment problem for the moments associated to the  sequence $a_n$ by \eqref{Szego} 
is determinate.
\item[(iii)] The unique measure $d\mu$ of the moment problem of (ii) is the probability measure proportional to $\vert \Gamma(\frac 14 + \frac 12 is)\vert ^2 ds$.
\item[(iv)]The representation $\sigma$ is 
the even component of the metaplectic representation $\varpi$ of $\widetilde{SL}(2,\R)$. 
\end{enumerate}
\end{thm}
\proof 
$(i)$~Using \eqref{diffsqan} together with $c=\frac 14$ one gets
$$
2 (a_{n-1}^2- a_n^2)=-4 n-1
$$
so that one obtains  \eqref{coefffixed}.\newline
$(ii)$~By \cite[Thm. 2, p. 86]{simon} {\em the Hamburger moment problem} corresponding 
is determinate, as follows from \cite[Corollary 6.19]{sm} since
$$
\sum_{n=0}^\infty \frac{1}{a_n} \equiv \sum_{n=0}^\infty \frac{2}{\sqrt{(2n+1)(2n+2)}}  = \infty .
$$ 
$(iii)$~The coefficients $a_n$ agree with those given by  \cref{fourierfmuall}, after rescaling by powers of $i$.\newline
$(iv)$~The eigenvalues of $\sigma(k)$, which serve as
weights for the representation $\sigma$ relative to the Cartan subgroup 
$\widetilde{S0}(2) \subset \widetilde{SL}(2,\R)$, exactly
reproduce the spectrum of the Hermite operator transported from $L^2(\R)^{ev}$. 
This identifies the irreducible representation $\sigma$ as being
the even component of the metaplectic representation $\varpi$ of $\widetilde{SL}(2,\R)$.\endproof 
We now  determine explicitly the moments $c_n:=\int t^nd\mu(t)$ where $d\mu$ is the probability measure proportional to $\vert \Gamma(\frac 14 + \frac 12 is)\vert ^2$. We use the equalities \eqref{Szego}, 
$$
k_n=(D_{n-1}/D_n)^{1/2}, \ \ a_n=\frac{k_n}{k_{n+1}}, \ \ a_n^2=\frac{1}{4} (2 n+1) (2 n+2),
$$
which imply
$$
D_{n-1}/D_n=k_n^2, \ \ k_n^2=k_0^2\ \prod_0^{n-1}a(j)^{-2} = \frac{4^n}{\Gamma (2 n+1)}
$$
This then determines the moments and shows that they are rational numbers, the first are
$$
c_0=1,\ c_2= \frac{1}{2}, \ c_4= \frac{7}{4},\ c_6= \frac{139}{8},\ c_8= \frac{5473}{16},\ c_{10}= \frac{357721}{32}
$$
What one finds in general is that the denominators of the moments are $2^n$ for $c_{2n}$  so that they are determined by a sequence of integers. The first ones are 
$$
1,\ 7,\ 139,\ 5473,\ 357721,\ 34988647,\ 4784061619,\  871335013633,\ 203906055033841,$$ $$59618325600871687,\ 21297483077038703899,\ 9127322584507530151393, \ldots $$
This is a classical integer sequence, known as $A126156$ \cite{OEIS}. In fact one has
\begin{prop}
\label{moments} 
The normalized moments $c_n$ of the measure $(2 \pi)^{-3/2}\vert\Gamma(\frac 14+\frac 12 is)\vert^2ds$ are such that $c_n=0$ for $n$ odd and
\begin{equation}\label{moments1}
\sum c_n\frac{(ix)^n}{n!}=\sqrt{\frac{2}{e^{x}+e^{-x}}}.
\end{equation}	
\end{prop}
\proof
The normalization means $c_0=1$ and that the measure $d\mu$ proportional to $\vert\Gamma(\frac 14+\frac 12 is)\vert^2ds$ is a probability measure. The left hand side of \eqref{moments1} is then equal to 
$$
\sum c_n\frac{(ix)^n}{n!}=\int_{-\infty}^\infty\exp(isx)d\mu(s)=\langle \exp(ixh)\fourier_\mu(w(\xi))\vert \fourier_\mu(w(\xi))\rangle
$$
where $\xi\in L^2(\R)^{ev}$, $\xi(x)=e^{-\pi  x^2}$, and $h$ is the  self-adjoint operator of multiplication by  $s$.
$$
\fourier_\mu(w(\xi))(s)=\int_0^{\infty}x^{1/2-is} e^{-\pi  x^2}d^*x=\frac 12 \pi ^{-\frac 14 +i\frac s 2} \Gamma \left(\frac 14 -i\frac s 2\right)
$$
One has $(\fourier_\mu f_\lambda)(s)=\lambda^{is}(\fourier_\mu f)(s)$ where $f_\lambda(u):=f(\lambda u)$, and taking $\lambda=e^x$ one obtains 
$$
\sum c_n\frac{(ix)^n}{n!}=\langle \lambda^{1/2}\xi_\lambda\vert \xi\rangle=\lambda^{1/2}\int_{-\infty}^\infty e^{-\pi  \lambda^2t^2}e^{-\pi  t^2}dt=
\lambda^{1/2}(1+\lambda^2)^{-1/2}
$$
which gives the required equality. \endproof\


\end{document}